\documentclass[12pt]{amsart}

\numberwithin{equation}{section}

\usepackage{amsmath}
\usepackage{amssymb}
\usepackage[colorlinks]{hyperref}
\usepackage{cleveref}
\usepackage{tikz}
\usepackage{mathabx}
\usepackage{xparse}
\usepackage[
    linecolor=black,
    colorinlistoftodos,
    tickmarkheight=.5ex,
    prependcaption
]{todonotes}

\usetikzlibrary{quotes}
\usetikzlibrary{calc}
\usetikzlibrary{cd}
\usetikzlibrary{patterns}
\usetikzlibrary{decorations.pathmorphing}
\usetikzlibrary{decorations.markings}
\usetikzlibrary{arrows}
\usetikzlibrary{arrows.meta}
\usetikzlibrary{bending}
\usetikzlibrary{shapes}

\tikzset{
    2-cell/.style={
        pattern=north west lines,
        opacity=0.5,
    }
}
\tikzstyle{in-equation} = [baseline=(current bounding box.center)] 

\tikzset{
    partial ellipse/.style args={#1:#2:#3}{
        insert path={+ (#1:#3) arc (#1:#2:#3)}
    },
}

\tikzset{%
    /cylinder settings/.is family,
    /cylinder settings,
        x radius/.store in=\cylinderradiusx,
        y radius/.store in=\cylinderradiusy,
        top/.store in=\cylindertop,
        bottom/.store in=\cylinderbot,
    default/.style={%
        x radius=2,
        y radius=1,
        top=3,
        bottom=0,
    },
}

\tikzset{%
    /configuration settings/.is family,
    /configuration settings/.search also={/tikz},
    /configuration settings,
        x radius/.store in=\configradiusx,
        y radius/.store in=\configradiusy,
        start/.store in=\configstart,
        end/.store in=\configend,
        line/.initial = {very thick, line cap=round},
    default/.style={%
        x radius=2,
        y radius=1,
    },
}

\tikzset{%
    /flow settings/.is family,
    /flow settings/.search also=/cylinder settings,
    /flow settings,
        top start/.store in=\flowtopstart,
        top end/.store in=\flowtopend,
        bottom start/.store in=\flowbotstart,
        bottom end/.store in=\flowbotend,
        top start/.value required,
        top end/.value required,
        bottom start/.value required,
        bottom end/.value required,
    default/.style = {
        /cylinder settings/default,
    }
}
\tikzset{%
    /twisted flow settings/.is family,
    /twisted flow settings/.search also={/flow settings, /cylinder settings, /tikz},
    /twisted flow settings,
        out top/.store in=\flowtwistedouttop,
        out bottom/.store in=\flowtwistedoutbottom,
        in top/.store in=\flowtwistedintop,
        in bottom/.store in=\flowtwistedinbottom,
        behind/.initial={draw={dashed}, fill, fill opacity=0.25},
    default/.style={%
        /flow settings/default,
    },
}

\tikzset{%
    pics/cylinder/.style={%
        code = {%
            \tikzset{%
                /cylinder settings,
                default,
                #1,
            }
            \edef\cylinderradius{\cylinderradiusx\space and \cylinderradiusy}

                \draw (0, \cylindertop) ellipse (\cylinderradiusx{} and \cylinderradiusy);

            \draw (0, \cylinderbot) [partial ellipse=180:360:{\cylinderradius}];
            \draw[dashed] (0, \cylinderbot) [partial ellipse=0:180:{\cylinderradius}];
            \draw (-\cylinderradiusx, \cylindertop) -- (-\cylinderradiusx, \cylinderbot);
            \draw ( \cylinderradiusx, \cylindertop) -- ( \cylinderradiusx, \cylinderbot);
        },
    },
    pics/configuration/.style={%
        code = {%
            \tikzset{%
                /configuration settings,
                default,
                #1,
            }

            \edef\configradius{\configradiusx\space and \configradiusy}

            \path[
                pic actions,
                very thick,
                line cap=round,
            ] (0, 0) +(\configstart:\configradius)
                arc (\configstart:\configend:\configradius)
                \ifx\relax\tikzpictext%
                \else%
                    node[midway, above, \tikzpictextoptions] {\tikzpictext}
                \fi%
                ; 

        },
    },
    pics/flow/.style={%
        code = {%
            \tikzset{%
                /flow settings,
                default,
                #1,
            }

            \begin{scope}
                \edef\flowradius{\cylinderradiusx\space and \cylinderradiusy}

                \path[pic actions]
                    (0, \cylindertop) +(\flowtopstart:\flowradius)
                        to[out=270, in=90]
                    ($(0, \cylinderbot) + (\flowbotstart:\flowradius)$)
                        arc (\flowbotstart:\flowbotend:\flowradius)
                        to[out=90, in=270]
                    ($(0, \cylindertop) + (\flowtopend:\flowradius)$)
                        arc (\flowtopend:\flowtopstart:\flowradius)
                        --
                    cycle
                    ;  
            \end{scope}
        },
    },
    pics/twisted flow/.style={%
        code = {%
            \tikzset{%
                /twisted flow settings,
                default,
                #1,
            }

            \begin{scope}
                \edef\flowradius{\cylinderradiusx\space and \cylinderradiusy}

                \path[pic actions]
                    (0, \cylindertop) +(\flowtopstart:\flowradius)
                        to[out=270, in=90]
                    (\cylinderradiusx, \flowtwistedoutbottom)
                        --
                    (\cylinderradiusx, \flowtwistedouttop)
                        to[out=90, in=270]
                    ($(0, \cylindertop) + (\flowtopend:\flowradius)$)
                        arc (\flowtopend:\flowtopstart:\flowradius)
                        -- cycle
                    ;  

                \path[pic actions, /twisted flow settings/behind]
                    (\cylinderradiusx, \flowtwistedouttop)
                        to[out=270, in=90]
                    (-\cylinderradiusx, \flowtwistedintop)
                        --
                    (-\cylinderradiusx, \flowtwistedinbottom)
                        to[out=90, in=270]
                    (\cylinderradiusx, \flowtwistedoutbottom)
                        -- cycle
                    ;  

                \path[pic actions]
                    (0, \cylinderbot) +(\flowbotstart:\flowradius)
                        to[out=90, in=270]
                    (-\cylinderradiusx, \flowtwistedinbottom)
                        --
                    (-\cylinderradiusx, \flowtwistedintop)
                        to[out=270, in=90]
                    ($(0, \cylinderbot) + (\flowbotend:\flowradius)$)
                        -- cycle
                    ;  
            \end{scope}
        },
    }
}

\pgfdeclarelayer{2-cells}
\pgfsetlayers{2-cells,main}

\theoremstyle{definition}

\newtheorem*{theorem}{Theorem}
\newtheorem{proposition}{Proposition}
\newtheorem{corollary}[proposition]{Corollary}
\newtheorem{lemma}[proposition]{Lemma}
 
\newtheorem{assumption}{Assumption}
\newtheorem{definition}[proposition]{Definition}
\newtheorem*{question}{Question}

\newtheorem{remark}[proposition]{Remark}
\newtheorem{example}[proposition]{Example}

\numberwithin{proposition}{subsection}

\newcommand{\pp}{\sqSubset}
\newcommand{\ppp}{\llcurly}

\newcommand{\C}{{\mathcal{C}}}
\newcommand{\cS}{{\mathcal{S}}}
\renewcommand{\O}{{\mathcal{I}}}

\newcommand{\G}{\mathcal{G}} 

\newcommand{\simpl}{\mathbf{\Delta}} 
\newcommand{\cycl}{\mathbf{\Lambda}} 
\newcommand{\slice}{{\mathcal{C}/T}}

\newcommand{\F}{I_{\C/T}}
\newcommand{\Aut}{{\mathrm{Aut}(T)}}
\newcommand{\aut}{{\mathbf{Aut}}}

\newcommand{\mfld}{{\mathbf{Mfld}}}

\newcommand{\Para}{{\cycl^{\!\!\infty}}}
\newcommand{\dupl}{{\mathbf{K}}}

\newcommand{\gcat}{{\G{\operatorname{-\mathbf{Cat}}}}}
\newcommand{\psh}{{\mathbf{Psh}}}
\newcommand{\preord}{{\mathbf{Pro}}}
\newcommand{\giso}{{\G{\operatorname{-\psh}}}}
\newcommand{\gpreord}{{A{\operatorname{-\preord}}}}

\newcommand{\source}{\mathsf{s}}
\newcommand{\target}{\mathsf{t}}

\DeclareMathOperator{\id}{id}
\DeclareMathOperator{\ho}{{\operatorname{ho}}}
\DeclareMathOperator{\track}{track}

\NewDocumentCommand{\Hom}{O{\C}}{#1}

\newcommand*{\To}{\mathrel{\Rightarrow}}

\newcommand*{\hcomp}{\hspace{0pt}}
\newcommand*{\WhiskL}{\hspace{0pt}}
\newcommand*{\WhiskR}{\hspace{0pt}}

\newcommand*{\hsim}{\sim_h}
\newcommand*{\horiz}{\diamond}

\newcommand*{\gact}{\mathbin{\smalltriangleright}}
\newcommand*{\ad}{\mathsf{a}}
\newcommand*{\adw}{\mathbin{\blacktriangleright}}

\begin{document}
\title[Cyclic duality for 
slice 2-categories]
{Cyclic duality for 
slice\\
and orbit 2-categories}
\author[Boiquaye]{John Boiquaye}
\author[Joram]{Philipp Joram}
\author[Krähmer]{Ulrich Kr\"ahmer}
\email{jboiquaye@ug.edu.gh}
\email{philipp.joram@taltech.ee}
\email{ulrich.kraehmer@tu-dresden.de}

\begin{abstract}
The self-duality of
the paracyclic category
is extended to a
certain class of
homotopy categories of 
(2,1)-categories. These 
generalise the orbit category of a
group 
and are
associated to certain self-dual preorders 
equipped with   
a presheaf of groups and a cosieve.
Slice 2-categories of
equidimensional submanifolds of a  
compact manifold
without boundary form a particular
case, and for $S^1$,
one recovers cyclic duality. 
This provides in
particular a visualisation of
the results of Böhm and \c Stefan
on the topic. 
\end{abstract}
\maketitle

\tableofcontents

\section{Introduction}
\subsection{Slice 2-categories}
This article is about 
embeddings of subobjects,
their deformations, and complements. 
Our basic setting
is:

\begin{assumption}\label{eins}
$\C$ is a (2,1)-category 
all of whose 1-cells are
monic.
\end{assumption}

Recall that 
the \emph{slice category} of $\C$ over an
object $T \in \C_0$ 
is the preorder of all
1-cells in $\C$ with codomain $T$ and the 
preorder relation 
$$
    x \preceq y 
    :\Leftrightarrow 
    \exists f \in \C_1 : 
    x = yf.
$$ 
By definition, a \emph{subobject}
of $T$ is an isomorphism 
class $[x]$ of an object in this
preorder, see
e.g.~\cite[Section V.7]{MR1712872}
or~\cite[Section 6.1.6]{MR2522659}. 

The \emph{slice 2-category} $\C/T$
records further information
about the subobjects of $T$: 
its objects are
1-cells with codomain $T$, 
and its 1-cells $x \to y$ are
2-cells 
$ \phi \colon x \To z$ in $\C$
with $z \preceq y$; think of a
deformation of a subobject $[x]$
to a subobject $[z]$ contained 
in $[y]$. Its
2-cells are 2-cells in $\C$ that
deform the target object $[z]$ 
inside $[y]$
(see Definition~\ref{s2cdef}), 
so in the
\emph{homotopy category}
$\ho(\slice)$
(Definition~\ref{hcdef}) 
such final perturbations   
of the target $[z]$ get
identified. 

\subsection{Orbit 2-categories}
For many $\C$, the 
ordinary slice category over $T$ is
self-dual, with the dual
$x^\degree$ of an object
representing some form of
complement of $[x]$ in $T$. The
question we are interested in is:

\begin{question}
When does a self-duality of 
$(\C/T_0,\preceq)$ lift to 
$\ho(\C/T)$?
\end{question}

This was triggered by the following
example that we will return to in
Section~\ref{motivation} below,
where we provide details and
definitions:

\begin{example}\label{urbsp}
In the (2,1)-category 
$\mfld^1$ of
embeddings of compact 1-dimensional
manifolds
(Definition~\ref{defmfld}), 
all ordinary slice categories are
self-dual. The homotopy category 
$\ho(\mfld^1/[0,1])$ is a model of
the \emph{simplicial category}
hence is not
self-dual. In contrast, 
$\ho(\mfld^1/S^1)$ is a model of
the \emph{paracyclic category} 
which is self-dual.
\end{example}

The answer we give assumes that 
the self-duality of
$(\C/T_0,\preceq)$ is equivariant
with respect to the natural action of the 
group $\Aut$ of
invertible 1-cells $T\to T$:

\begin{assumption}\label{zwei}
$\degree \colon
    \C/T_0 \to \C/T_0$,
$x \mapsto x^\degree$
is a map such that 
$$
    [x^{\degree\degree}] = [x],
    \quad
    [(gx)^\degree] = 
    [g(x^\degree)],\quad
    x \preceq y 
    \Leftrightarrow 
    y^\degree \preceq x^\degree
$$
holds for all $x,y \in \C/T_0$ 
and $g \in \Aut$.
\end{assumption}

Such a self-duality gives rise to 
a subrelation 
$\ppp$ of $\preceq$ which is 
a \emph{$\Aut$-cosieve} in
$(\C/T_0,\preceq)$, i.e.~which is closed
under the $\Aut$-action and under
postcomposition
(Definition~\ref{defcos}, 
Proposition~\ref{pppiscos}).
In $\C=\mfld^d$, $x \ppp y$
means that $[x]$ is contained in the 
interior of $[y]$. 

Our main result 
states that if all $x \in \C/T_0$ 
satisfy a strong form of the
\emph{homotopy extension property} 
(Assumptions~\ref{seven}
and~\ref{fortytwo} in
Section~\ref{assumptionssec}, where
this is discussed in full detail) 
and admit an abstract version
of \emph{tubular
neighbourhoods}
(Assumption~\ref{last} therein), 
then the answer to
our question is affirmative. 
More precisely, we prove the
following theorem, where
expressions such as
$ \gamma y$ and $y \xi$ denote the 
horizontal composition of the
2-cell $ \mathrm{id} _y$ with
2-cells $ \gamma $ respectively $
\xi $, and where 
$$
    G:=\bigcup_{g \in \Aut} 
    \C_2(\mathrm{id} _T,g).
$$
 
\begin{theorem}
If $\ppp$ is an $\Aut$-cosieve in 
$(\C/T_0,\preceq)$ with 
$$ 
    \mathrm{id} _T^\degree \ppp
\mathrm{id}_T^\degree,$$ 
and if 
for all $f,h \colon X \to Y$, $y
\colon Y \to T$, and 
$ \phi \colon yf \To yh$,
we have
\begin{align*}
    (\exists \xi \colon f \To h 
    : 
    \phi = y \xi) 
    \quad\Leftrightarrow\quad
    (\forall u \ppp y^\degree 
    \exists \gamma \in G 
    : \gamma u=u,
    \gamma yf = 
    \phi)
\end{align*}
and
for all $u \ppp y^\degree,
    v \ppp y^\degree$ there exists  
   $\tau \colon 
    \mathrm{id}_T \To t$ in $G$ and  
    $r \ppp y$ with
\begin{equation}\label{apricot}
    \tau u=u, \ \tau v = v,\   
    [tr] = [y],
\end{equation}
then $\degree$ lifts to an 
$\Aut$-equivariant self-duality on $\ho(\C/T)$.
\end{theorem}

This applies in particular to 
$\mfld^d/T$ if
$T$ is a 
manifold with empty boundary 
(Examples~\ref{polityka},
\ref{polityka2},
\ref{tubular}, and
\ref{hepgood}).

We present  
the above theorem as a special
case of a more general self-duality
result: let $\G$ be a (strict)
2-group, $A$ be the group of its
1-cells, and $G$ be its source
group. Then we
associate  
a $\G$-category 
$\O_s$ to any 
$A$-preorder 
$(S,\sqsubseteq)$ equipped
with an $A$-equivariant presheaf 
$s \colon x \mapsto G_x$ 
of subgroups $G_x \subseteq G$,
(Proposition~\ref{o2cat}). 
When $S$ is the poset of all
subgroups of $A=G$ itself 
and $ s $ is the
identity, then $\O_s$ is the dual 
of the 
\emph{orbit category} of $G$ 
(Example~\ref{santoe}, see 
\cite{MR889050} for more
information). When the
$A$-preorder is self-dual and $\pp$
is an $A$-cosieve satisfying
(\ref{apricot}), then $\O_s$ gets
upgraded to a (2,1)-category with a
self-dual homotopy category
(Proposition~\ref{modest},  
Corollary~\ref{modest2}). The
remaining assumptions of
our theorem are there to
imply $\ho(\C/T) \cong 
\ho(\O_{s_{\C/T}})$ as
$\aut(T)$-categories, where 
$\G=\aut(T)$ is the automorpism
2-group of $T$
(Corollary~\ref{modest2},
Example~\ref{tubular}). 

We are not aware of a reference
that considers this exact type of
(2,1)-category, and we only develop
the part of the theory required to
prove the above theorem. Studying
other examples and applications 
might therefore be
an interesting topic for future
research, as there are many
applications of classical 
orbit categories in equivariant
algebraic topology, see
e.g.~\cite{MR656126,MR293618,MR1621979}
and in particular 
\cite {MR690052}, or the more
algebraically motivated articles 
\cite{MR3178068,MR3338608,
MR2457810,MR3576335,MR1852376}, as
well as \cite{MR3415508}.

\subsection{Cyclic duality}\label{motivation}
The motivation for this article
lies in homological  
and homotopical 
algebra. The categories of chain
complexes and of simplicial objects 
are not self-dual -- 
applying contravariant functors
yields cochain complexes
respectively cosimplicial objects. 
So by construction, the category of 
chain \emph{and} cochain
complexes (same graded module, no  
compatibility between boundary and
coboundary map assumed) is self-dual.
Building on the seminal work of
Connes (\cite{MR777584}, see also
\cite[Appendix 3.A]{MR1303779}),
Dwyer and Kan \cite{MR826872}
extended the Dold-Kan
correspondence to this setting and
called the corresponding 
homotopical objects and the 
governing index category $\dupl$
\emph{duplicial}.
The self-duality 
of $\dupl$ also descends to 
Connes' \emph{cyclic
category} $ \cycl $ which is
a quotient, and extends 
to the 
\emph{paracyclic category}
$\Para$
which is a localisation:
 
\begin{definition}\label{paradef}
The categories  
$ \simpl \subset 
\dupl \subset \Para $ are 
defined as follows:
\begin{itemize}
\item The objects are the 
natural numbers 
$0,1,2,3,\ldots$.
\item The morphisms 
$f \colon n \rightarrow m$ are 
the maps 
$\mathbb{Z} \rightarrow
\mathbb{Z} $ satisfying 
\begin{enumerate}
\item $i \le j \Rightarrow 
f(i) \le f(j)$ for all $i,j$,
\item $ f(j+n+1) = f (j)+m+1$ for all 
$j$,
\item $f(0) \ge 0$ (in $\dupl$ and 
$\simpl$), 
\item $f(n) \le m$ (in $\simpl$).
\end{enumerate} 
\end{itemize}
The \emph{cyclic dual} of $f \colon
n \to m$ is the morphism 
$f ^\degree \colon m \to n$ given
by
$$
    f^\degree (i) := 
    \max \{j \mid 
    -f(-j) \le i\}.
$$
\end{definition}
See~\cite{MR777584,MR826872,MR1244971,MR923136, 
MR1143017} for some original
references for these definitions,
and \cite[Appendix B]{MR3904731} for
some recent reflections. Note
that the above self-duality is 
not the most studied one on $\Para$,
but one that restricts to 
$\dupl$, see e.g.~\cite[p585]{MR826872} and
\cite[Section~4.2]{MR2803876} -- this
also addresses the warning on
\cite[p381]{MR3904731}.

Böhm and \c Stefan
\cite{MR2956318}
explored the self-duality of 
$\Para$ further 
from the perspective of
the bar construction. Our focus is
different: we describe 
$ \Para $ and $ \simpl$
in a unified way in which we
can point exactly at the reason why
the one is self-dual and the other
is not: both 
are (skeletal subcategories of) 
$\ho(\C/T)$ for suitable $\C$ and
$T$. For $\simpl$ we are looking at 
$\mfld^1/[0,1]$, while for $\Para$ it is 
$\mfld^1/S^1$, and the latter is
self-dual as $S^1$ has empty
boundary. 

We also find the resulting
visualisation of $\Para$ 
clarifying in several ways. 
In the standard description, 
the object $n$ of $\Para$ gets
visualised as $n+1$ points
on $S^1$. We replace these by
tubular
neighbourhoods, which is
anyway natural in many settings,
e.g.~the study of the cyclic homology
of DG algebras. Furthermore, 
the fact that
the objects of $\Para$ are
self-dual, $n^\degree=n$, 
is seen to be in a sense
coincidental -- the
complement of $n+1$ intervals in
$S^1$ happens to be again $n+1$
intervals. These are isotopic to the
original ones, but in higher-dimensional manifolds $T$, $[x]$
and the closure $[x^\degree]$ 
of the complement of 
$ \mathrm{im}\, x$ are in general
not diffeomorphic. 

Most importantly to us, this provides a spatial 
view on the results of
Böhm and \c Stefan. Their main aim
was to conceptually explain cyclic
duality in the setting of
Hopf-cyclic (co)\-ho\-mo\-lo\-gy
\cite{MR2065371,MR2815134,MR1718047,MR1868538}.
They showed (see
\cite[Theorem~4.7]{MR2956318}) 
that the simplicial
object resulting
from the bar construction
associated to a comonad $S_l$ 
and coefficients that they denote
by $\sqcup,\sqcap$ 
becomes
paracyclic in the presence of a  
second comonad $S_r$, a comonad
distributive law $ \Psi \colon S_l
S_r \To S_r S_l$, and 
$ \Psi $-(op)coalgebra structures
$i,w$ on the coefficients.
In our visualisation, this
corresponds to connecting the end
points of the interval
$[0,1]$ in which the simplicial
object is realised using a second
interval to obtain a circle $S^1$,
and the second comonad $S_r$ lives on
the dark side of the moon. 
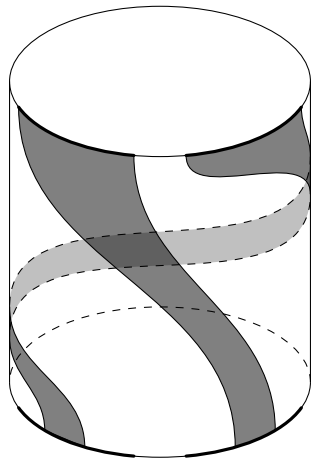
\begin{figure}
\begin{tikzpicture}
    \def\rx{2};
    \def\ry{1};
    \def\r{{\rx} and {\ry}}
    \def\h{4};

    \colorlet{embx}{black}
    \colorlet{flow}{black}

    \path (0, 0) pic {cylinder={top=\h, x radius=\rx, y radius=\ry}};

    \begin{scope}[fill=flow, fill opacity=0.5]
        \pic at (0, 0) [draw, fill] {
            twisted flow={
                x radius=\rx,
                y radius=\ry,
                top=\h,
                top start=280,
                top end=340,
                bottom start=220,
                bottom end=240,
                out top=3,
                out bottom=2.5,
                in top=1,
                in bottom=0.6,
                behind/.style = {dashed, fill opacity=0.25},
            },
        };

        \pic at (0, 0) [draw, fill] {
            flow={
                top=\h,
                top start=200,
                top end=260,
                bottom start=300,
                bottom end=320
            },
        };
    \end{scope}

    \pic[draw, embx] at (0, \h)
        {configuration={start=200, end=260}};
    \pic[draw, embx] at (0, \h)
        {configuration={start=280, end=340}};

    \pic[draw, embx] at (0, 0)
        {configuration={start=200, end=260}};
    \pic[draw, embx] at (0, 0)
        {configuration={start=280, end=340}};
\end{tikzpicture}
\caption{The cyclic operator $2
    \to 2$ in $\ho(\mfld^1/S^1)$.}
    \label{fig:twist-crossing-over}
\end{figure}

The string diagrams in
\cite{MR2956318} can now be
seen as
planar projections of our depiction
of morphisms in $\Para$ in which we
draw the track of a point in $S^1$
under an isotopy on a cylinder. 
In particular, 
Figure~\ref{fig:twist-crossing-over}
depicts the cyclic operator 
$f \colon \mathbb{Z} \to
\mathbb{Z}$, $f(j) = j+1$, acting
on the object $2 \in \Para$ (which
is denoted 
$t_2$ in \cite{MR2956318}).
The 
natural transformation $w$ 
corresponds to the track of an
interval passing on the 
right end from the front to
the back of the cylinder, 
the distributive law 
$ \Psi $
corresponds to two 
tracks crossing in the planar 
projection, 
while in the spatial resolution 
one of them runs down 
the front of the cylinder 
and the other one the back,
and finally the natural
transformation $i$ is where the
track from the back reappears on
the left end of the cylinder.
\\

The remainder of this article is
divided into three sections. In the
first, we provide some definitions 
from the theory of slice
2-categories, and discuss the
example $\mfld^d$ of embeddings of
compact $d$-dimensional manifolds. 
In the second, we develop parts of
a general theory of orbit
2-categories associated to certain
presheaves of groups 
on a preorder. Ordinary 
slice categories are 
used throughout as a
guiding example, but we also
discuss a few group-theoretic
examples in order to demonstrate
the scope of the concepts. In the
final section we focus again
on $\C/T$ and discuss the
assumptions of
our main theorem in detail.
\\
 
Throughout the paper, we suppress all
set-theoretic problems, so we tacitly
assume all categories to be as small as
required. Readers who are concerned about 
the application to manifolds should
restrict to submanifolds of 
$ \mathbb{R} ^{2d}$. Similarly, we
focus on strict 2-categories.    
\\

\noindent {\bf Acknowledgements.}
J.B.~is supported by
UG Carnegie BANGA-Africa Project
and
P.J.~is supported by
Estonian Research Council grant PSG749.

\section{Slice
2-categories}\label{slice2catsec}
This section contains background
material on (2,1)-categories (see 
e.g.~\cite{MR4261588} for
further information).
We recall in particular  the
definition of the homotopy
category 
$\ho(\slice)$
of a 
slice 2-category and discuss 
in some detail the example
of the (2,1)-category $\mfld^d$ 
of embeddings of 
compact $d$-dimensional 
manifolds with 2-cells given by 
isotopies. 

\subsection{(2,1)-categories}
Throughout, $\C$ is a strict 
(2,1)-category, that is, the
composition of 1-cells as
well as both the vertical and the
horizontal composition of 2-cells is 
strictly associative, 
and all 2-cells are 
invertible. In addition, we assume
that all 1-cells are monic.
The set of
1-cells between objects $X,Y
\in \C_0$ is denoted by 
$\C_1(X, Y)$; 1-cells are
denoted by lower case Roman letters
and their composition is written
as concatenation.  
The set of
2-cells between $f,g \in
\C_1(X,Y)$ is denoted by
$\C_2(f,g)$; 2-cells are
denoted by lower case Greek letters, and
their
vertical respectively horizontal
compositions by 
$ \alpha \circ\beta $ respectively 
$ \alpha \hcomp \beta $.
The identity 1-cell
in $\C_1(X,X)$ is denoted by $\id_X$;
analogously, the
identity 2-cell in 
$\C_2(f,f)$ is denoted by $\id_f$.
We denote the source and target maps 
$\C_2 \to \C_1$ and 
$\C_1 \to \C_0$ both by 
$\source$ respectively $\target$.  
Horizontal composition of 2-cells
with identity 1-cells is called
\emph{left-} respectively
\emph{right-whiskering}, and we
write $f \WhiskL \xi :=
\id_f \hcomp \xi$ respectively
$\xi \WhiskR g := \xi\id_g$,
if no confusion arises.

\begin{remark}\label{2groupremark}
Let
$ \alpha \colon f \To g$ be a
2-cell and $ \alpha^* \colon g
\To f$ be its inverse, 
so $ \alpha \circ \alpha^* = 
\mathrm{id} _g$ and 
$\alpha^* \circ \alpha = \mathrm{id}
_f$. If in addition $f,g \colon X \to
Y$ are
invertible as 1-cells, then 
a straightforward
computation shows 
(see e.g.~\cite[{Lemma 7}]{MR4190076})
that $ \alpha $ 
also has a horizontal 
inverse given by 
$$ 
	\alpha^{-1} := g^{-1} \WhiskL
\alpha^*
	\WhiskR f^{-1} \colon
	f^{-1} \To g^{-1}
$$ 
and 
the vertical inverse of $
\alpha^{-1} $
is $ (\alpha ^{-1})^* := 
g^{-1} \WhiskL \alpha 
\WhiskR f^{-1} \colon g^{-1} 
\To f^{-1}$. 
If furthermore 
$ \delta \colon g \Rightarrow h$
is another 2-cell, then we have in
this case
\begin{align}\label{pform}
    \delta \circ \alpha
&= 
    (\delta \circ \alpha) 
     \id_{\mathrm{id} _X}\nonumber\\
&= (\delta \circ \alpha)  
    ((\alpha^{-1})^* 
    \circ \alpha^{-1}) 
    \id_f\nonumber\\
&= ((\delta  
    (\alpha^{-1})^*) 
    \circ (\alpha \alpha ^{-1}) 
    )\id_f \\
&= ((\delta (\alpha^{-1})^* ) 
    \circ 
    \id_{\mathrm{id}_Y}) \id_f
\nonumber\\
&= \delta  (\alpha^{-1})^*   
    \id_f = 
    \delta \mathrm{id} _{g^{-1}}
    \alpha .\nonumber 
\end{align}
An alternative way to carry out
such computations is offered by the 
graphical calculus of string diagrams:
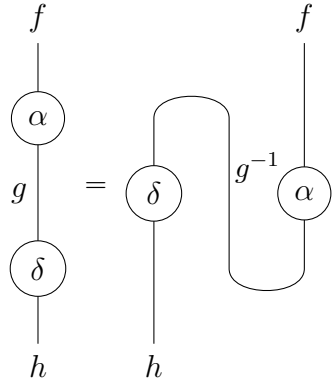
\begin{equation*}
    \begin{tikzpicture}[x=2cm, y=-4cm, baseline=($(alpha.base)!.5!(delta.base)$)]
        \coordinate (alpha) at (0, 1/4);
        \coordinate (delta) at (0, 3/4);

        \draw[black]
            (alpha |- 0, 0)
            node[above] {$f$}
            --
            node[midway, left] {$g$}
            (delta |- 0, 1)
            node[below] {$h$}
            ;
        \node[draw, circle, fill=white] at (alpha) {$\alpha$};
        \node[draw, circle, fill=white] at (delta) {$\delta$};
    \end{tikzpicture}
    \;=\;
    \begin{tikzpicture}[x=2cm, y=-4cm, baseline=($(alpha.base)!.5!(delta.base)$)]
        \coordinate (alpha) at (1, 1/2);
        \coordinate (delta) at (0, 1/2);
        \coordinate (center) at (1/2, 1/2);
        \coordinate (cap) at (1/2, 1/4);
        \coordinate (cup) at (1/2, 3/4);

        \draw[black] (alpha |- 1, 0)
            node[above] {$f$}
            -- (alpha |- cup)
            to[out=south, in=south] (center |- cup)
            to[out=north, in=south]
                node[above right,
font=\small] {$\!g^{-1}$}
                (center |- cap)
            to[out=north, in=north] (delta |- cap)
            to (delta |- 0 , 1)
            node[below] {$h$}
            ;

        \node[draw, circle, fill=white] at (alpha) {$\alpha$};
        \node[draw, circle, fill=white] at (delta) {$\delta$};
    \end{tikzpicture}
\end{equation*}
\end{remark}

\subsection{Slice categories}
In ordinary category theory,
the slice category $\slice$
over some object $T$ of a category
$\C$ has as objects all morphisms
$x\colon X \to T$, and as morphisms
commutative triangles
\begin{equation*}
    \begin{tikzcd}[column sep=small]
        X
            \ar[rr, "f"]
            \ar[dr, "x"{swap}]
            &
            &
        Y
            \ar[dl, "y"]
            \\
            &
        T
            &
    \end{tikzcd}.
\end{equation*}

For (2,1)-categories $\C$, there is 
the following generalisation
in which the equality
$x = y f$ is replaced with a
2-cell $x \To y f$:

\begin{definition}\label{s2cdef}
    The \emph{slice 2-category}
    over an object $T \in \C_0$ is
    the (2,1)-category $\slice$ with
    the following data:
    \begin{itemize}
        \item The objects are 
            1-cells
            $x\colon X \to T$ of
            $\C$,
        \item 1-cells between
            objects $x \colon 
X \to T$ and
            $y \colon Y \to T$ are pairs
            $(f, \phi)$ consisting
            of a 1-cell
            $f\colon X \to Y$ and
            a 2-cell
            $\phi\colon x \To y
 f$ that we will usually
depict as 
            \begin{equation}\label{eq:slice-2-cat:1-cell}
                \begin{tikzcd}[column sep=small]
                    X
                        \ar[dr,"x"{swap, name=x}]
                        \ar[rr, "f"]
                        &
                        &
                    Y
                        \ar[dl, "y"]
                        \ar[
                            from=x,
                            Rightarrow,
                            shorten <=1.7ex,
                            shorten >=1.5ex,
                            "\phi"{description}
                        ]
                        \\
                        &
                    T
                        &
                \end{tikzcd}
            \end{equation}
 \item The composition of
            1-cells
            $\phi\colon x \To y f$
            and
            $\psi\colon y \To z g$
            is given by
            $(g, \psi) (f, \phi)
                := (gf,
                    \psi \WhiskR f \circ \phi)$.
        \item 2-cells between
            $\phi\colon x \To y
 f$ and
            $\psi\colon x \To y
 g$ are
            2-cells
            $\xi\colon f \To g$
            such that
            \begin{equation}
                \label{eq:slice-2-cat:2-cell-cond}
                \psi = y\WhiskL \xi \circ \phi
            \end{equation}
 \item Vertical and horizontal
            composition of 2-cells
            in $\C / T$ is defined
            as the vertical
            (respectively
            horizontal) composition
            in
            $\C$.
 \end{itemize}
\end{definition}
           
See \cite[{Definition 7.1.1(3)}]{MR4261588} for
                a diagrammatic
depiction of the
                \emph{ice cream
                cone condition}
                \eqref{eq:slice-2-cat:2-cell-cond}.

\begin{remark}\label{cancel}
As we focus on 
(2,1)-categories in which
all 1-cells $x \colon X \to T$
are monic, the 2-cell $ \phi $
which is part of a 
1-cell $ (f,\phi) \colon 
x \to y$ in $\C/T$
uniquely determines the 1-cell 
$f$. We denote the
latter by $ f_\phi $ and simply
write $ \phi $ for 
$(f_\phi ,\phi )$.  
\end{remark}

\subsection{Homotopy categories}

Instead of just forgetting the
2-cells, one can construct
an ordinary category out of a
(2,1)-category $\C$
by identifying 1-cells if they
are related by a 2-cell:

\begin{definition}\label{hcdef}
    We call
    $f, g \in  \C_1$
\emph{homotopy equivalent} and
write $f \hsim g$ if there
    exists a 2-cell
    $\alpha \colon f \To g$.
    The \emph{homotopy
    category} $\ho(\C)$ of 
$\C$ is the category 
with objects ${\ho(\C)}_0 := \C_0$ and
morphisms
    \begin{align}
        {\ho(\C)}_1(X, Y) &:=
            \C_1(X, Y) / \hsim.
    \end{align}
\end{definition}

Note that \cite{MR4261588} calls $\ho(\C)$
the classifying category of $\C$,
see Example 2.1.27 therein. 

\begin{remark}\label{subobjects}
As we assume all 1-cells in $\C$
to be monic, the objects of $\C/T$
represent the \emph{subobjects} of
$T \in \C_0$. By definition, these
are equivalence classes $[x]$ 
with 
$x \colon X \rightarrow T,y \colon
Y \rightarrow T$ being equivalent
if and only if there is an
invertible 1-cell $f \colon X
\rightarrow Y$ such that $x = yf$. 
A \emph{homotopy
equivalence} between objects 
$X,Y \in \C_0$ is by definition a
pair of 
1-cells $f \colon X \rightarrow Y,g
\colon Y
\rightarrow X$ whose classes in  
$\ho(\C)$ are inverses of each
other, $fg \hsim \mathrm{id} _Y,
gf \hsim \mathrm{id} _X$. Thus the
isomorphism classes of the objects
in $\C/T$ are  
the \emph{homotopy classes 
of subobjects} of $T$.
\end{remark}

\subsection{Embeddings of manifolds}
We now define the motivating
example   
for this paper:

\begin{definition}\label{defmfld}
By the (2,1)-category $\mfld^d$ of 
\emph{embeddings of compact 
$d$-dimensional
manifolds}  
we mean the following:
\begin{itemize}
\item The objects of $\mfld^d$ are
compact smooth
$d$-dimensional manifolds $X$
with (possibly empty) boundary 
$ \partial X$, together with 
the empty manifold
$\emptyset$.  
\item
A 1-cell in
$\mfld^d$ is an 
embedding, by which we mean a 
smooth injective
immersion.
\item The composition of 1-cells is the
ordinary composition of maps. 
\item A 2-cell  
$f \To g$ 
between embeddings $f,g \colon
X \to Y$ 
is an isotopy class $[\phi]$ of isotopies 
$ \phi $ from $f$ to $g$, that is, 
of smooth maps $\phi\colon [0, 1]
    \times X \to Y$ 
such that the restrictions 
$ \phi (t,-) \colon X \to Y$ are
embeddings and for some 
$ \varepsilon >0$, we have
$$
\phi(t,-) = f,\quad
\phi(1-t, -) = g\quad \forall t \in
    [0,
    \varepsilon].
$$
\item The
horizontal composition of 2-cells is
induced by the level-wise composition of
isotopies, 
$$
	(\alpha \hcomp \beta )(t,p):=
	\alpha (t,\beta (t,p))
$$
while the vertical composition is induced
by the concatenation of the path
$\beta (-,p)$ followed by the path 
$ \alpha (-,\beta
(1,p))$. The vertical inverse of a
2-cell is taken by inverting the
orientation of a path. 
\end{itemize}
\end{definition}

See
e.g.~\cite[{p.~111}]{MR1336822} for
further details. 
Note that we
do not make any additional
assumptions on the behaviour of 
embeddings 
on $ \partial X$; in
particular, we do not assume it
embeds $X$ as a neat submanifold 
in the sense of 
\cite[{p.~30}]{MR1336822}.
Note further that the vertical
composition of isotopies themselves is
not strictly associative; however, since  
we define 2-cells to be 
isotopy classes of
isotopies, $\mfld^d$ is indeed a
strict 2-category. 

\subsection{Submanifolds}\label{rain}
The slice 2-category $\mfld^d/T$ 
describes embeddings of 
manifolds
into an ambient manifold $T$ of the
same dimension $d$. For
this entire Section~\ref{rain}, 
we fix embeddings 
$x \colon X \to T$
and $y \colon Y \to T$. 

A 1-cell in $(\mfld^d/T)_1(x,y)$ is
represented by an isotopy 
$$ \phi \colon x \To y  f_\phi $$
in $\mfld^d$, where 
$ f_\phi \colon X \to Y$ is the unique embedding
such that 
$$ 
	\phi (1,-) = y 
f_\phi.
$$ 
Note that we
are in the situation of
Remark~\ref{cancel}.  

To visualise such 1-cells,
it is convenient to introduce their 
track:

\begin{definition}
The \emph{track}
of an isotopy $ \phi $ 
is the smooth map 
$$
\track(\phi) \colon 
	[0,1] \times X
	\to 
	[0,1] \times T,\quad
	(t,p) \mapsto 
	(t, \phi (t,p)).
$$
\end{definition}
\begin{example}\label{exampleabove}
Figure~\ref{fig:mfld-slice:1-cell}
depicts the track of an isotopy 
$ \phi $ which represents a 1-cell 
$ [\phi] \in (\mfld^1/S^1)_1(x,y)$. Recall that we work
with isotopy classes of isotopies as
2-cells in $\mfld^d$ rather than
isotopies themselves. The isotopies
in the class $[\phi]$ share the embeddings $x$ and
$y$ of $X$ respectively $Y$ into
$S^1$ at the top
$(t=0)$ respectively
bottom $(t=1)$ of the cylinder.
If $ \psi \in [\phi ]$ is another
representative, then 
$\track(\psi)$ differs 
for $0<t<1$ from 
$\track( \phi )$ 
by an isotopy   
$$
	\omega \colon 
	[0,1] \times [0,1] \times 
	S^1 \to 
	[0,1] \times 
	S^1,
$$
so we have 
$$
	\track(\phi)(t,p) = 
	(t,\omega (0,t,p)) ,
	\quad
	\track(\psi)(t,p) = 
	(t,\omega (1,t,p)),
$$
as well as 
$$
	\omega(s,0,p) = 
	x(p),\quad 
	\omega(s,1,p) = y(p).
$$
	
The vertical composition of 1-cells
in $\mfld^1/S^1$ can be visualised as stacking such
cylinders on top of each other.

\begin{figure}
    \begin{tikzpicture}[
            configuration line/.style = {very thick},
        ]

        \pic {cylinder={x radius=2, y radius=1, top=3}};

        \foreach \i/\topstart/\topend/\botstart/\botend/\pos in {
            0/210/235/200/225/above right,
            1/260/285/270/290/above,
            2/295/325/310/330/above left%
        } {
            \pic[draw, fill=black, fill opacity=0.25] at (0, 0) {
                flow={
                    top=3,
                    top start=\topstart, top end=\topend,
                    bottom start=\botstart, bottom end=\botend,
                }
            };

            \pic[draw, very thick, line cap=round] at (0, 3) {
                configuration={
                    start=\topstart, end=\topend,
                }
            };
        }

        \foreach \i/\start/\stop in {0/200/225, 1/260/340} {
            \pic[draw, very thick, line cap=round] at (0, 0) {
                configuration={
                    start=\start, end=\stop,
                }
            };
        }

        \draw[->, shorten <=1em, shorten >=1em]
            (3, 3)
            node {$x$}
            --
            node[midway, right] {$[\phi]$}
            (3, 0)
            node {$y$};
    \end{tikzpicture}

    \caption{
        The track of an isotopy
        representing a 1-cell
        $[\phi] \colon x \to y$
        of $\mfld^1 / S^1$.
        $X=\source (x)$ 
    consists of three
        copies of the interval $[0,
        1]$, $Y=\source(y)$ of two.
        The thick lines at the top
        and bottom mark the subsets
        $\{0\} \times 
    \mathrm{im}\, x$ and
        $\{1\} \times 
    \mathrm{im}\, y$ of
        $[0, 1] \times S^1$.
    }
    \label{fig:mfld-slice:1-cell}
\end{figure}
\end{example}

Assume now that $ [\phi] , [\psi] 
\colon x \to y$ are two 1-cells
in $\mfld^d/T$, and let 
$ f_\phi,f_\psi$ be
the underlying embeddings of $X$
into $Y$.
A 2-cell
$[\xi]$ in
$(\mfld^d/T)_2([\phi],[\psi])$ 
is by definition a 2-cell $[\xi] \colon f_\phi
\To f_\psi$ in $\mfld^d$, so 
the representative 
$\xi \colon [0,1] \times X
\rightarrow Y$ is an isotopy from 
$f_\phi $ to $f_\psi$ satisfying 
$\psi = (y \WhiskL \xi ) \circ \phi $.

\begin{example}
As in Example~\ref{exampleabove} above, we
consider 
$d=1$ and $T=S^1$. Then the action
of a 2-cell $\xi$ can be pictured
as in
Figure~\ref{fig:mfld-slice:2-cell}.
\begin{figure}
    \begin{tikzpicture}[
            configuration line/.style = {very thick},
            in-equation,
        ]

        \pic {cylinder={x radius=2, y radius=1, bottom=0, top=5}};

        \foreach \i/\topstart/\topend/\corstart/\corend in {
            0/210/235/210/220,
            1/260/285/270/280,
            2/295/325/300/335%
        } {
            \pic[draw, fill=black, fill opacity=0.25] at (0, 0) {
                flow={
                    top=5,
                    bottom=0,
                    top start=\topstart,
                    top end=\topend,
                    bottom start=\corstart,
                    bottom end=\corend,
                }
            };

            \pic[draw, very thick, line cap=round] at (0, 5) {
                configuration={start=\topstart, end=\topend}
            };
        }

        \foreach \i/\start/\stop in {0/200/225, 1/260/340} {
            \pic[draw, very thick, line cap=round] at (0, 0) {
                configuration={
                    start=\start, end=\stop,
                }
            };
        }

        \draw[->, shorten <=1em, shorten >=1em]
            (-3, 5)
            node {$x$}
            --
            node[midway, left] {$[\psi]$}
            (-3, 0)
            node {$y$};
    \end{tikzpicture}%
    \hspace{.5em}=\hspace{.5em}%
    \begin{tikzpicture}[
            configuration line/.style = {very thick},
            in-equation,
        ]

        \pic {cylinder={x radius=2, y radius=1, bottom=0, top=5}};
        \draw (0, 2) [partial ellipse=180:360:{2 and 1}];
        \draw[dashed] (0, 2) [partial ellipse=0:180:{2 and 1}];

        \foreach \i/\topstart/\topend/\botstart/\botend/\corstart/\corend in {
            0/210/235/200/225/210/220,
            1/260/285/270/290/270/280,
            2/295/325/310/330/300/335%
        } {
            \pic[draw, fill=black, fill opacity=0.25] at (0, 0) {
                flow={
                    top=5,
                    bottom=2,
                    top start=\topstart, top end=\topend,
                    bottom start=\botstart, bottom end=\botend,
                }
            };

            \pic[draw, fill=black, fill opacity=0.25] at (0, 0) {
                flow={
                    top=2,
                    bottom=0,
                    top start=\botstart,
                    top end=\botend,
                    bottom
                    start=\corstart,
                    bottom end=\corend,
                }
            };

            \pic[draw, very thick, line cap=round] at (0, 5) {
                configuration={start=\topstart, end=\topend}
            };
        }

        \foreach \i/\start/\stop in {0/200/225, 1/260/340} {
            \pic[draw, very thick, line cap=round] at (0, 2) {
                configuration={
                    start=\start, end=\stop,
                }
            };
            \pic[draw, very thick, line cap=round] at (0, 0) {
                configuration={
                    start=\start, end=\stop,
                }
            };
        }

        \draw[->, shorten <=1em, shorten >=1em]
            (3, 5)
            node {$x$}
            --
            node[midway, right] {$[\phi]$}
            (3, 2)
            node {$y$};
        \draw[->, shorten <=1em, shorten >=1em]
            (3, 2)
            --
            node[midway, right]
                {$y \WhiskL [\xi]$}
            (3, 0)
            node {$y$};
    \end{tikzpicture}

    \caption{
        The isotopy
        $\xi\colon [0, 1] \times X \to Y$
        between $f_\phi$ and $f_\psi$
        represents a 2-cell in
        $(\mfld^1/S^1)_2([\phi],
[\psi])$.
    }
    \label{fig:mfld-slice:2-cell}
\end{figure}
We stress that the action of 2-cells is
given by the vertical composition with
1-cells that are not arbitrary but 
have to be of the form $y\WhiskL\xi$.
In
Figure~\ref{fig:mfld-slice:2-cell} this means that
for all possible choices of 
$\xi$, the (grey) track of
$y\WhiskL\xi$
will stay within 
$\mathrm{im}\, y \subset S^1$, it
can not freely use all of $S^1$.

One observes by direct inspection
that the paracyclic category 
$ \Para$ (Definition~\ref{paradef}) 
with an initial and a terminal
object added can be realised as a
skeletal subcategory of 
$\ho(\mfld^1/S^1)$:
the object 
$n$ of 
$\Para$ can be identified with any
embedding of $n+1$ intervals into $S^1$,
say
$$
	x_n \colon 
	\bigcup_{j=0}^{n}
	[j,j+1/2]
    \rightarrow 
	S^1, \quad 
    t \mapsto 
	\exp\left(\frac{ 2 \pi i t}{n+1}
    \right).
$$
 
An isotopy 
$\phi$
that represents a 1-cell
$x_n \rightarrow x_m$
in $\mfld^1/S^1$ defines unique 
smooth maps 
$\phi_j \colon 
	[0,1]  
	\to \mathbb{R}$
with 
$$
	\phi (t,x_n(j)) =
	\exp
   (2 \pi i \phi_j(t)),\quad
    \phi_j(0) = \frac{j}{n+1},
    \quad
    j=0,\ldots,n. 
$$
Now  
$ \phi (1,-) \preceq x_m$ 
implies  
that there is a unique morphism 
$f \colon n \to m$ in 
$\Para$ such that  
$$
	\phi_j(1) \in	
	\left[
	\frac{f(j)}{m+1},
	\frac{f(j)+1/2}{m+1} \right],
$$
and the assignment $ \phi \mapsto
f$ induces  
an isomorphism between 
the full subcategory of 
$\ho(\mfld^1/S^1)$ consisting of 
all $x_n$, $0 \le n<\infty$, and 
$\Para$.  
\end{example}

\begin{example}
Analogously, together with the
empty embedding and $ \mathrm{id}
_{[0,1]}$, the simplicial
category $ \simpl $ can be realised
as a skeletal subcategory of 
$\ho(\mfld^1/[0,1])$. 
\end{example}

\begin{example}
When 
$d=3$ and $T=S^3$, then
embeddings of
a solid 3-torus are knots, and the
1-cells between them are given by
isotopies.
\end{example}
\begin{example}\label{annulus}
Consider the 2-dimensional
manifolds 
\begin{align*}
&  
    X := \{ 
    \left({a \atop b}\right)
    \in \mathbb{R} ^2
    \mid \sqrt{ 
    a^2+b^2} \le 1/5\},\\
 & 
    Y := \{ 
    \left({a \atop b}\right)
    \in \mathbb{R} ^2
    \mid  
    1/2 \le \sqrt{a^2+b^2} \le 1\},\\
 & T := \{ 
    \left({a \atop b}\right)
    \in \mathbb{R} ^2
    \mid  
    \sqrt{a^2+b^2} \le 1\} 
\end{align*}
and the embeddings 
\begin{align*}
    f &\colon X \to Y,\quad  
    \left({a \atop b}\right)
    \mapsto 
    \left({a \atop b+3/4}\right),
    \\
    y &\colon Y \to T,\quad
    \left({a \atop b}\right)
    \mapsto
    \left({a \atop b}\right),\\ 
    x &:= y  f \colon X \to T.
\end{align*}
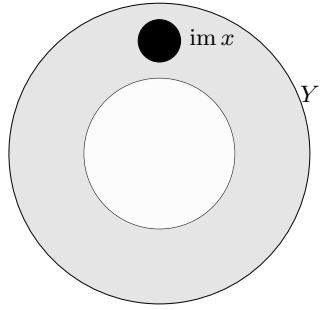
\begin{figure}[h]
\begin{tikzpicture}
   \begin{scope}[thick,font=\scriptsize][set layers]
        \end{scope}
    \draw[solid] (0,-1) circle (1);
    \draw[solid] (0,-1) circle (2);
    \path [draw=none,fill=gray, fill opacity = 0.2] (0,-1) circle (2);
    \path [draw=none,fill=white, fill opacity = 0.9] (0,-1) circle (1);
    \filldraw[black] (0,0.5) circle (8pt);

     \begin{scriptsize}    
\draw[color=black] (0.7,0.55) node
{$\mathrm{im}\, x$};    
\draw[color=black] (2,-0.2) node {$Y$};    
\end{scriptsize} 
\end{tikzpicture}
\caption{A disc embedded in an annulus}
\end{figure}

The
identity 2-cell $
\mathrm{id}_x \colon 
x \To x$ in $\mfld^2$ yields a
1-cell $(f, 
\mathrm{id}_x) \colon x \to y$ 
in the slice
2-category $\mfld^2/T$ and there   
is a nontrivial 2-cell 
$\xi \colon f \To f$ in 
$\mfld^2$ 
which moves the small disc
(resp.~its embedding) once round
the annulus without rotating the
disc itself,
$$
    \xi(t,-) := 
    R_t f R_{-t}
    \colon X 
    \to Y,\quad
    R_t(
    \left( {a \atop b}\right)) := 
    \left({\cos(2 \pi t)a + 
    \sin(2 \pi t)b \atop  
    -\sin (2 \pi t)a + \cos(2 \pi
t)b }\right).    
$$
As $y \WhiskL \xi$ is (isotopic to)
$ \mathrm{id} _x$ in $\mfld^2$, $\xi$ 
satisfies (\ref{eq:slice-2-cat:2-cell-cond})
(with $ \psi = \phi=\mathrm{id} _x$). 
It thus
defines 
a nontrivial 2-cell 
in
$(\mfld^2/T)_2((f,\mathrm{id}_x),(f,\mathrm{id}
_x))$.
This example also shows that 
even when $y$ is monic,
whiskering
on the left with $y$ is not
necessarily injective.
\end{example}

\section{Cyclic duality for orbit 2-categories}\label{4advent}
We now study a certain 
class of 2-thin (2,1)-categories
that generalise the
orbit category of a group, and
sufficient conditions under which 
their homotopy categories are
self-dual. In the subsequent
section, we will realise homotopy
categories of suitable 
slice 2-categories in this way
and thus prove our main theorem.

\subsection{2-groups and crossed
modules}\label{2groups}
Throughout this section, $\G$ denotes 
a (strict) \emph{2-group} (see
e.g.~\cite{MR2068521} for an
excellent account on the concept), 
that is, 
a (2,1)-category
with a single
object $T$ in which all 1-cells are
invertible. Recall that by the 
\emph{Brown-Spencer theorem} 
\cite{MR0419643}, such a 2-group
can be equivalently described as
follows:

\begin{definition}
A \emph{crossed module}  
is a pair of group homomorphisms 
$$
    \target
    \colon G \to A,\quad
    \ad \colon A \to
    \mathrm{Aut}(G)
$$ 
such that for all $h \in A$ and 
$ \gamma, \alpha \in G$, we have 
$$
    \target (\ad (h) (\gamma )) = 
    h \target ( \gamma ) h^{-1},
$$
that is, $\target$ is $A$-equivariant
with respect to the action $\ad$ on
$G$ and the adjoint action on
itself, and the \emph{Peiffer identity}
\begin{equation}\label{peiffer}
 \gamma \alpha \gamma ^{-1} = 
    \target(\gamma ) 
    \alpha \target(\gamma)^{-1}.
\end{equation}
\end{definition}

The crossed module that is
associated to (and describes)
$\G$ has
$$
    A:=\G_1(T,T),
$$
the group of 1-cells, and 
$$
    G:=\bigcup_{g \in A}
    \G_2(\mathrm{id} _T ,g),
$$ 
the so-called \emph{source
group} of $\G$ 
(with
horizontal composition as group
structure, cf.~ 
Remark~\ref{2groupremark}). 
The group homomorphism $\target$ is
given by the target map in $\G$,
and $\ad(h)$ is given by
left and right whiskering with $h$
respectively $h^{-1}$,
$$
    \ad(h)( \gamma ) := h \gamma
h^{-1}.
$$

\begin{example}\label{gsetpo}
Every group $A$ acts on itself by 
conjugation, and taking 
$\target := \mathrm{id} _A$ yields
a crossed module. The corresponding 
2-group has the group $A$ as its 
1-cells and 
for any $ g,h \in A$ a unique 2-cell 
$ g \To h $. 
\end{example}

\subsection{$A$-preorders and
$\G$-categories}\label{exc}
Recall that a \emph{preorder}
is a \emph{thin} category,
i.e.~one with at
most one morphism between any two
objects. We denote such a category
$\cS$ by $(S,\sqsubseteq)$, where
$S:=\cS_0$ is the set of objects 
and $\sqsubseteq$ is the reflexive
and transitive binary relation on
$S$ that a morphism $x \to y$
exists. 
We will write 
$$
    x \sim y :\Leftrightarrow 
    (x \sqsubseteq y \mbox{ and } y
    \sqsubseteq x)
$$ 
if $x,y \in S$ are isomorphic in $\cS$.

\begin{definition}
An \emph{$A$-preorder} is 
a preorder $(S,\sqsubseteq)$ with
an action of $A$ on $S$ such that 
$
    x \sqsubseteq y \Rightarrow 
    gx \sqsubseteq gy 
$
holds for all $x,y \in S$ and 
$ g \in A$. 
\end{definition}

\begin{example}
The set $U_G$ of all
subgroups of $G$ carries 
a natural action of $A$ and hence
defines an 
$A$-preorder
$(U_G,\supseteq)$.
\end{example}

\begin{definition}
By a \emph{$\G$-presheaf} on an
$A$-preorder $(S,\sqsubseteq)$ we
shall mean an $A$-equivariant map
of preorders 
$$
    s \colon (S,\sqsubseteq) 
    \to (U_G,\supseteq),
    \quad
    x \mapsto G_x,
$$
that is, a map 
such that for all $x,y \in S$ and 
all $g \in
A$, we have
\begin{equation}\label{aproug}
	G_{gx} =
	\ad(g)(G_x),
    \quad
    x \sqsubseteq y \Rightarrow 
	G_y \subseteq G_x.
\end{equation}
\end{definition}

Here is a toy example:

\begin{example}\label{gsetpo2}
Let $A=G$ be any group acting on
itself by conjugation
(Example~\ref{gsetpo}), let $S$ be any 
$A$-set, and 
$$
    A_x := \{ g \in A \mid gx = x\}
$$ 
be the isotropy group of $x
\in S$. Then $x \sqsubseteq y :
\Leftrightarrow A_y \subseteq A_x$
turns $S$ into a preorder,
and the $\{A_x\}$ form 
a $\G$-presheaf on it. 
\end{example}

This generalises as follows to
any action of a 2-groups on
a category:

\begin{example}\label{actionsof2groups}
Assume that a 
2-group $\G$
acts on a category 
$\cS$, that is, 
$g \in A$ acts by 
a functor $\cS \to \cS$,
and   
$ \gamma \colon h \To g$ in $\G_2$ 
by a natural transformation 
between the functors given by
$g,h$ 
(more abstractly, we are given
a 2-functor 
$\G \to \mathbf{Cat}$ which 
sends the unique object $T$ of 
$\G$ to $\cS$). If we denote the 
components of these natural
transformations by  
$ \gamma _x \colon hx \to
gx$,  
then 
$$
	G_x := \{ \gamma \in G \mid 
	\gamma _x = \mathrm{id}
	_x\},\quad
    x \sqsubseteq y :
\Leftrightarrow G_y \subseteq G_x
$$
turns $S:=\cS_0$ into an
$A$-preorder equipped with a 
$\G$-presheaf.
\end{example}

In these examples,  
$x \sqsubseteq y$ is  
defined as
$G_y \subseteq G_x$, but in general, 
it is a subrelation. The example
that we will
use to prove our main result
arises in this way as a subobject
of one of the above type:

\begin{example}\label{elmena}
Let $\C$ be a (2,1)-category,
$T \in \C_0$, and $\aut(T)$ be the
\emph{automorphism 2-group} of $T$,
that is, the
2-group of all invertible 1-cells 
$T \to T$ and of all 2-cells
between these. As a special case of
the preceding
Example~\ref{actionsof2groups}, 
$\aut(T)$ 
acts on the category 
underlying the slice 2-category 
$\C/T$,
with the action 
$gx$ of $g \in
A=\aut(T)_1$ on 
$x \in \C/T_0$ given by composition, 
and with 
$ \gamma _x := \gamma \WhiskR x$ ($
\gamma $ whiskered on the right by
$x$). Note that this
$\aut(T)$-action descends to 
$\ho(\C/T)$. As above, we set 
$G_x:= \{\gamma \in G \mid 
\gamma x=x\}$. 
In this way, we obtain 
an $\aut(T)$-presheaf 
$$
    s_{\C/T} \colon
    (\C/T_0, 
    \preceq) \to 
    (U_G,\supseteq)
$$
on the ordinary slice category of
$\C$ over $T$, so here 
$$
    x \preceq y
    :\Leftrightarrow  
    \exists f \in \C_1 : x =
yf. 
$$
We will show that in
good cases, $\ho(\C/T)$ only
depends on 
$s_{\slice}$. 
\end{example}

\begin{remark}\label{paint1}
To paint a more complete picture,
one could define a 2-category
$\gpreord$ of $A$-preorders and
a 2-category $\giso = \gpreord/U_G$
of $\G$-presheaves. 
Example~\ref{actionsof2groups} 
then becomes part of 
the definition of a 2-functor
$\gcat \to \giso$, where
$\gcat$ is the 2-ategory of all
categories with an action of $\G$. 
Note also that $\giso$ does not 
make use of the target map
$\target$ of $\G$, the construction
works for all actions of a group
$A$ on another group $G$. 
As this is not directly relevant
for our main problem, we decided
not to work this out here though. 
\end{remark}

\subsection{Orbit categories}
Let $s \colon (S,\sqsubseteq) 
\to (U_G,\supseteq)$ be 
a $\G$-presheaf.

\begin{remark}
The group $G$ acts 
via the target map 
$\target \colon G \to A$ on $S$,
and to avoid confusion, we denote
this action by 
$$
    \gamma \gact x :=
\target(\gamma ) x,\quad
    \gamma \in G,x \in S. 
$$
Note that the Peiffer identity
(\ref{peiffer}) implies 
\begin{equation}
    \label{peiffer2}
    G_{\gamma \gact x} = 
    G_{\target(\gamma )x} = 
    \ad(\target(\gamma )) (G_x) =
    \gamma G_x \gamma^{-1}. 
\end{equation}
Here and
later, we freely use that the 
product 
$$
	LR:=\{\alpha \beta \in G
	\mid \alpha \in L,\beta \in R\}
$$
turns the power
set $P(G)$ of $G$ into a monoid with 
unit element $\{1\}$. Note that 
subgroups of $G$ are idempotent
elements of $P(G)$.   
Expressions such as $ \gamma G_x
\gamma ^{-1}$ are a shorthand
notation for $ \{\gamma \} G_x 
\{\gamma ^{-1}\}$.  
\end{remark}
As we will explain in
Example~\ref{santoe} below, the 
following construction generalises 
the (dual of
the) \emph{orbit category} of a
group
(see e.g.~\cite[Section I.10]{MR889050}).
We will
denote the composition of morphisms
using
$\horiz$ to distinguish it
from the product of
elements or subsets of $G$ which we
continue to denote simply by
concatenation. Note also that 
$G/G_x$ refers to the set of cosets 
$ \gamma G_x$ of the subgroup
$G_x$, not a slice 2-category.

\begin{proposition}\label{o2cat}
The following defines a
$\G$-category $\O_s$:
\begin{itemize}
\item the set of objects is
$\O_{s,0}:=S$,
\item the morphism sets 
are $\O_{s,1}(x,y) := 
	\{\gamma G_x \in 
	G/G_x \mid 
	\gamma \gact x \sqsubseteq y\}
$, 
\item the composition is induced
by the product in $G$; that is, 
if $ \gamma G_x \colon x
\to y, \delta G_y \colon y
\to z$ are morphisms, then
$$
    (\delta G_y) \horiz (\gamma G_x) :=
    \delta \gamma G_x,
$$
\item the functor given by $
h \in A$ acts on objects via the
original action on $\O_{s,0} =S$ 
and on morphisms by 
$$
    \O_{s,1} (x,y) \to 
    \O_{s,1}(hx,hy) ,\quad
    \gamma G_x
    \mapsto 
    (h \gamma h^{-1} )
    G_{hx},
$$
\item and the natural transformation
assigned to  
$ \gamma \in G$ has 
components   
$\gamma_x := 
\gamma G_x
\colon x \to \gamma \gact x$.
\end{itemize}
\end{proposition}
\begin{proof}
Note
first that the composition 
is equal to the
multiplication of subsets of $G$:
$\gamma \gact x \sqsubseteq y$
implies 
$G_y \subseteq G_{\gamma \gact x} =
\gamma G_x \gamma ^{-1}$ so that 
$$ 
    \delta G_y \gamma G_x \subseteq 
\delta \gamma G_x \gamma ^{-1} 
\gamma G_x = \delta \gamma G_xG_x=
\delta \gamma G_x
    = 
    (\delta G_y) \horiz 
    (\gamma G_x). 
$$
The reverse
inclusion 
holds as $1 \in G_y$,
so $ \delta \gamma G_x \subseteq 
\delta G_y \gamma G_x$. 

It follows that the definition 
of $\delta G_y \horiz 
\gamma G_x$ is independent of the
choice of representatives $ \delta
, \gamma $ and also that 
$\horiz$ is associative. 

Furthermore, 
$ \delta \gact y \sqsubseteq z$ and 
$\gamma \gact x
\sqsubseteq y$ together
imply 
$ 
(\delta \gamma ) \gact x
\sqsubseteq \delta \gact y$, hence
$ 
(\delta \gamma ) \gact x
\sqsubseteq 
\delta \gact y 
\sqsubseteq z 
$. Thus 
$ \delta \gamma G_x$ is
a morphism $x \to z$ in $\O_s$. 

That the $\G$-action is
well-defined is verified
straightforwardly -- for the
naturality of the transformations $
\gamma _x$ use the Peiffer identity
(\ref{peiffer}); note also that 
$ \gamma _x$ is an isomorphism with
inverse $ \gamma ^{-1} G_{\gamma
\gact x}$. 
\end{proof}

As 
an immediate consequence of the
definition, we have:

\begin{corollary}
All $\gamma G_x \in \O_{s,1} (x,y)$
are monic.
\end{corollary}

\begin{example}\label{santoe}
Consider $G=A$ as in
Example~\ref{gsetpo} and 
$$ 
    s=\mathrm{id} _{U_A} \colon 
    U_A \to U_A,\quad
    x \mapsto A_x=x.
$$
In this case, $\O_{\mathrm{id}
_{U_A}}$ has the subgroups $x
\subseteq A$ as objects, and 
$$
    \O_{s,1}(x,y) = 
\{g x \in A/x \mid 
y \subseteq gxg^{-1}\}.
$$
Note 
$ gx $ is a coset of the subgroup
$x$. 
The \emph{orbit category} (see
e.g.~\cite[Section I.10]{MR889050}) of
$A$ is the category whose objects
are the coset spaces 
$A/x$, $ x \in U_A$, and whose
morphisms  
$g \colon A/y \to A/x$ are 
$A$-equivariant maps. Such a map
$g$ is
uniquely determined by its value on 
$y=1y \in A/y$, and an element 
$ g x \in A/x$ occurs as such
a value if and only if $ y
\subseteq gxg ^{-1}$.
Thus $\O_s$ is (isomorphic to) the
dual of the orbit category. 
\end{example}

Thus we define in full generality:

\begin{definition}
We call $\O_s^\degree$ the
\emph{orbit category} of 
$s$.
\end{definition}

We will show later that 
under the assumptions 
in our main theorem, the
category underlying $\C/T$ 
is isomorphic to
$\O_{s_{\C/T}}$ (recall 
Example~\ref{elmena}).
This is why we focus on $\O_s$
rather than its dual. 

\begin{remark}\label{paint2}
Expanding Remark~\ref{paint1}, 
the assignment $s \mapsto \O_s$ 
can be made part of a 2-functor 
$\giso \to \gcat$.
The construction
from Example~\ref{actionsof2groups} 
almost recovers $s$ from $\O_s$,
only the relation
$\sqsubseteq$ is replaced by the
potentially weaker one meaning 
$G_y \subseteq G_x$. So there is a
natural transformation from the
identity 2-functor on $\giso$ to the
composition of the two
constructions. As far as we can
see, this is in general not part of a
2-adjunction: starting with any
$\G$-category $\cS$ and defining 
$s$ as in
Example~\ref{actionsof2groups}, 
the morphisms in the resulting 
$\G$-category $\O_s$ are generated
by the $ \gamma _x$ together with
virtual embeddings $x \to y$ 
whenever $G_y \supseteq G_x$. These 
might not correspond to any actual
morphisms in $\cS$, and, conversely,
there might by morphisms in $\cS$
that are entriely unrelated to the
$\G$-action (e.g.~when $\G$ is
nontrivial but acts trivially on
$\cS$). However, our main result
roughly says that for 
$\cS=\C/T$ with the necessary  
transitivity conditions added in
the form of the homotopy extension
property discussed below, $\cS$ can be
reconstructed from $\O_s$. So we
believe there is a class of
well-behaved $\G$-categories for
which the two 2-functors form a
split adjunction of suitable
2-categories. 
\end{remark}

\subsection{Self-dual $A$-preorders}
The aim of the remainder of
Section~\ref{4advent} is to upgrade 
$\O_s$ to a (2,1)-category with a self-dual
homotopy category. 
In order to do so, we need to
assume the presence of two
additional structures on the 
underlying $A$-preorders. The
first one is an $A$-equivariant 
self-duality:

\begin{definition}
An \emph{$A$-self-duality} on an
$A$-preorder $(S,\sqsubseteq)$ 
is a map 
$$
	S \rightarrow S,
	\quad
	x \mapsto x^\degree
$$
such that for all $x,y \in S$ and
$g  \in A$, we have 
\begin{equation}\label{scholz}
	x \sim x^{\degree\degree},\quad
	x \sqsubseteq y \Leftrightarrow	
	y^\degree \sqsubseteq x^\degree,\quad
	(gx)^\degree \sim  
    g (x^\degree).
\end{equation}
\end{definition}

Here is the main 
example that we have in mind:

\begin{example}\label{polityka}
Consider  
$s_{\C/T}$ (Example~\ref{elmena}) 
with $\C=\mfld^d$, so
$(\C/T)_0$ consists of 
all embeddings
$x \colon X \to T$ of a manifold
$X$ into a compact 
manifold $T$ of the same
dimension $d$. The preorder
relation $ x \preceq y
\Leftrightarrow \exists f : x=yf$ 
means that 
$\mathrm{im}\, x
\subseteq \mathrm{im}\, y$, and 
if $T$ has empty boundary, then 
the inclusion $x^\degree$ of the 
closure of the
complement $T\setminus
\mathrm{im}\, x$ into $T$ is 
a $\mathrm{Diff}(T)$-self-duality.
\end{example}

Bear in mind though that our
setting is quite general.
In particular, $x^\degree$ is
in general not a
complement in most of the standard meanings
of the word. Here is a somewhat
fake example which shows amongst
other things that  
$x$ and $x^\degree$ do  
not need to be jointly epic:

\begin{example}
Let $\C$ be the (2,1)-category
whose objects are the intervals 
of the form $(-\infty,t]$ and 
$[s,\infty)$, $s,t \in \mathbb{R}
$, plus $\emptyset$ and
$T=\mathbb{R} $,
whose 1-cells are inclusions
(so $\C$ is a preorder
and in fact a poset), and all of
whose 2-cells
are identities (so $G=A$ is trivial).
Then 
$$
    [s,\infty)^\circ :=
(-\infty,s-1],\quad
    (-\infty,t] ^\circ := 
    [t+1,\infty)
$$  
and
$$
    [s,\infty)^\bullet :=
(-\infty,s+1],\quad
    (-\infty,t] ^\bullet := 
    [t-1,\infty)
$$  
both define self-dualities on the
preorder 
$(\C/T_0,\preceq)$, but 
we have 
$$ 
    \mathrm{im}\, x \cup
    \mathrm{im}\, x^\circ \neq 
\mathbb{R} ,\quad   
    \mathrm{im}\, x \cap
\mathrm{im}\, x^\bullet \neq
\emptyset.
$$ 
\end{example}

Finally, here are two examples of
a very different nature:

\begin{example}
Let $G=A$ be any group, 
viewed as a 2-group as in
Example~\ref{gsetpo}, and 
let $S$ be any $A$-set. If $H \lhd
G$ is a normal subgroup, setting $
G_x := H$ for all $x \in S$ and
$x \sim y$ for all $x,y$ defines a
$\G$-presheaf, and 
$x^\degree:=x$ is an $A$-self-duality. 
\end{example}

\begin{example}
Let again $G=A$ be any group, $H,K$
be normal subgroups, and 
$S:=\{H,K \}$ with trivial
$G$-action, and set $G_x:=x$,
$x \sqsubseteq y :\Leftrightarrow 
G_x \supseteq G_y$. Then
$H^\degree:=K,K^\degree:=H$ defines 
an $A$-self-duality.
\end{example}

In this example, $G_{x^\degree}$ is not
necessarily contained in the
centraliser $Z_G(G_x)$, but 
note that we always have:

\begin{lemma}\label{dieampel}
If $s$ is a $\G$-presheaf and 
$ \degree $ is a self-duality
on the underlying $A$-preorder,  
then we have 
$$
    N_G(G_x) = N_G(G_{x^\degree}).
$$ 
In
particular, 
$G_x \subseteq N_G(G_{x^\degree})$.
\end{lemma}
\begin{proof}
By (\ref{scholz}) and
(\ref{peiffer2}), $G_x =
G_{\gamma \gact x}
$ implies 
\begin{equation*}
G_{x^\degree} = 
G_{(\gamma \gact x)^\degree} =G_{\gamma
\gact x^\degree}.\qedhere 
\end{equation*}
\end{proof}
\begin{corollary}\label{hast}
$G_xG_{x^\degree} =
G_{x^\degree}G_x$ is a subgroup of $G$. 
If, in addition, we have $G_x \cap
G_{x^\degree}=\{1\}$, then 
$G_x G_{x^\degree} \cong 
G_x \times G_{x^\degree}$. 
\end{corollary}

\subsection{$A$-cosieves}
The second ingredient we use for
the (2,1)-upgrade of $\O_s$ is 
a \emph{cosieve} in the 
$A$-preorder underlying $s$ 
(recall that a cosieve in a category is 
just a set of morphisms closed
under postcomposition with any
morphism). This provides an
abstract concept of an ``interior''
of a subobject; like the
self-duality it should 
be compatible the $A$-action:

\begin{definition}
A binary relation $\pp$ is 
an \emph{$A$-cosieve} in 
the $A$-preorder $(S,\sqsubseteq)$ 
if for all 
$x,y,z \in S$ with $x \pp y$ and 
all $ g \in A$, we have
$$
    x \sqsubseteq y,\qquad   
    gx \pp 
    gy,\qquad
    y \sqsubseteq z
    \Rightarrow x \pp z.
$$
\end{definition}

Note that this implies:

\begin{lemma}\label{lucy}
If $y \sim z$, then
$ x \pp y \Leftrightarrow 
x \pp z$.
\end{lemma}

Again, we first consider the
example that motivates the
definition:

\begin{example}\label{polityka2}
Consider  
$s_{\mfld^d/T}$ 
(Example~\ref{elmena},
Example~\ref{polityka}). 
Then the relation $ x \ppp y$ 
($x \colon X \to T,y \colon Y \to
T$ embeddings of 
manifolds) that 
$ \mathrm{im}\, x$ is contained in
the interior of $ \mathrm{im}\, y$
(i.e.~the boundary of 
$\mathrm{im}\, x$ does not
intersect the boundary of 
$ \mathrm{im}\, y$) is a
$\mathrm{Diff}(T)$-cosieve. 
\end{example}

\begin{example}\label{legon}
Any self-duality 
defines a cosieve 
$$ 
    x \pp_\degree y : \Leftrightarrow  
   (x \sqsubseteq y \mbox{ and }
     \langle G_{y^\degree} \cup G_x \rangle
=G),
$$
where the right hand side denotes
the fact that $G_{y^\degree}$ and
$G_x$
together generate $G$ as a group.
In the preceding 
Example~\ref{polityka2} that we are
most interested in, this
is, however, a stricter
relation than $\ppp$: if 
$ \mathrm{im}\, x$ is not properly
contained in the interior of 
$ \mathrm{im}\, y$, then it has a
nontrivial intersection with 
$\mathrm{im}\, y^\degree$, the
closure of the complement of 
$ \mathrm{im}\, y$ in $T$. Thus
this intersection contains at least
one point $p$ which is fixed
by all elements of 
$\langle G_{y^\degree} \cup 
G_x \rangle$ and in particular by
their codomains $ g \colon T \to T$. 
 However, for each point in a
manifold there is some
diffeomorphism that is isotopic to
$\mathrm{id} _T$ and that moves this
point, so 
$ \langle G_{y^\degree} \cup 
G_x \rangle \subsetneq G$.
Therefore,  we have 
$ x \pp_\degree y \Rightarrow x
\ppp
y$, but the converse does not hold
in general.   
In particular,
if $\mathrm{im}\, x \subset T=S^1$ is a
nonempty proper submanifold, then 
$G_x$ consists of isotopy classes
of isotopies $ \gamma \colon 
\mathrm{id} _{S^1} \To g \in 
\mathrm{Diff}(S^1)$ with 
$ \gamma \WhiskR x = \mathrm{id}
_x$, so $ g $ belongs 
to the subgroup
$\mathrm{Diff}(S^1)_x$ of
diffeomorphisms that restrict
to the identity on $ \mathrm{im}\,
x$. Since the complement of $
\mathrm{im}\, x$ is contractible and
we consider isotopy classes of
isotopies rather than isotopies
themselves, such $ \gamma $ are
uniquely determined by $g$. That
is,  
$ \target \colon G_x \to \mathrm{Diff}(S^1)_x$
is a group isomorphism.
So if $x \ppp y$ for $ \mathrm{im}\,
y \subsetneq S^1$, then 
$ \langle G_{y^\degree} \cup 
G_x \rangle $ 
does not contain the 2-cell 
$ \mathrm{id} _{S^1} 
\To \mathrm{id} _{S^1}$ that is
represented by the isotopy
\begin{equation}\label{the2cell}
    [0,1] \times 
    S^1 \to S^1,\quad
    (t,e^{2 \pi i s}) \mapsto 
    e^{2 \pi i (s+t)}
\end{equation}
that rotates the entire circle
once, so we do not have 
$x \pp_\degree y$.
\end{example}

Here is an abstract algebraic
example that illustrates that the
behaviour of $ \pp_\degree $ 
can be very different
from what one might expect:

\begin{example}
Let $A=G$ be a group viewed as a
2-group (Example~\ref{gsetpo}) that 
acts trivially
on a set $S$. If $\{A_x\}_{x \in
S}$ is a family of normal subgroups 
with $A_x
\subseteq A_y \Leftrightarrow x=y$,
then we obtain an $A$-preorder with
a $\G$-presheaf 
by setting 
$ x \sqsubseteq y
:\Leftrightarrow x=y$, and an
$A$-self-duality by setting
$x^\degree:=x$ for
all $x$. 
Assume furthermore that for any 
$x \neq y$, $\langle A_x \cup A_y
\rangle = A$. As a concrete 
example, we can take
$A=\mathbb{Z}$, $S$ the set of
prime numbers, and $A_x= \langle x
\rangle $ the group of all
integers divisible by $x$. 
Then there are no $x,y \in S$ with 
$ x \pp_\degree y$ at all. 
\end{example} 

Before we move on to the
construction of the (2,1)-category 
$\O_s$, we briefly discuss for 
the example
$(\C/T_0,\preceq)$ the 
close relation between
$\mathrm{Aut}(T)$-cosieves 
in the $\mathrm{Aut}(T)$-preorder and
cosieves in $\C$ itself:
  
\begin{proposition}
If $\ppp$ is an 
$\mathrm{Aut}(T)$-cosieve in
$(\C/T_0,\preceq)$, then 
\begin{align*}
    S_{\ppp} 
&:= \{
    f \in \C_1 \mid 
    \forall y \in \C_1 :
    \source(y) = \target (f)
    \Rightarrow 
    yf \ppp y\},\\
    \bar S_\ppp
&:= \{ hf \mid 
    h \colon Y \to Z,
    f \colon X \to Y, 
    \mbox{ and } \exists y
    \colon Y \to T : yf \ppp y\}. 
\end{align*}
are cosieves $S_\ppp \subseteq \bar
S_\ppp$ in $\C$. Conversely,
if $S$ is any cosieve in $\C$, then 
$$
    x \ppp_S y : \Leftrightarrow 
    \exists f \in S : x = yf 
$$ 
is an $\mathrm{Aut}(T)$-cosieve, 
and we have 
$S_{\ppp_S}=\bar S_{\ppp_S} = S$ as well as
$$
x \ppp_{S_{\ppp}} y \Rightarrow x
\ppp y \Rightarrow 
    x \ppp_{\bar S_{\ppp}} y.
$$
\end{proposition}
\begin{proof}
Let $\ppp$ be an $\mathrm{Aut}(T)$-cosieve. Given 
$ f \colon X \to Y$ in $S_\ppp$ 
and  
$h \colon Y \to
Z$, we need to show 
$hf \in S_\ppp$, that is, that for all
$z \colon Z \to T$
we have $zhf \ppp z$. To see this,
set 
$y:=zh,x:=yf=zhf$. Then 
on the one hand, we have 
$x = yf \ppp y$ 
by the definition of $S_\ppp$, and 
on the other hand
$y=zh \preceq z$ by the definition
of $\preceq$. Since $\ppp$ is a 
cosieve, this implies  
$x \ppp z$ as required. That $\bar
S_\ppp$ is a cosieve is immediate
(it is the cosieve generated by
all $f$ with $yf \ppp y$ for
some $y$).  

Conversely, if $S$ is a cosieve and 
$ x \ppp_S y$, 
then evidently 
$x \preceq y$ and 
$ gx \ppp_S 
gy$ (as $gx=gyf 
\Leftrightarrow x=yf$ for all
$g \in\mathrm{Aut}(T)$). Also
if
$y \preceq z$ with  
$y = zh$ for some $h \colon Y \to
Z$, then $x = yf=zhf$ and as 
$hf \in S$ ($S$ is a
cosieve), $x \ppp_S z$. So $\ppp_S$ is
an $\mathrm{Aut}(T)$-cosieve.  

We have
$$
    S_{\ppp_S}
    =
    \{ f \colon 
    X \to Y \mid 
    \forall y \colon 
    Y \to T \exists 
    \tilde f \in S : yf = y\tilde
f\}
$$
and as all 1-cells are monic, we
evidently have 
$S_{\ppp_S}=S$. Similarly,
$\bar S_{\ppp_S}=S$. 
Conversely, 
$x \ppp_{S_\ppp} y \Rightarrow x
\ppp y$ follows immediately from  
\begin{equation*}
    x \ppp_{S_{\ppp}} y 
    \Leftrightarrow 
    \exists f : (x = yf \mbox{ and }
    \forall \tilde y \colon 
    Y \to T : \tilde y f \ppp 
    \tilde y).
\end{equation*}
The implication 
$x \ppp y \Rightarrow 
x \ppp_{\bar S_\ppp}$ follows
analogously from 
\begin{equation*}
    x \ppp_{\bar S_\ppp} y 
    \Leftrightarrow 
    \exists f : (x = yf \mbox{ and }
    \exists \tilde y,h,g : \tilde y g \ppp 
    \tilde y,f = hg),
\end{equation*}
just take $g=f,\tilde
y=y,h=\mathrm{id} _{\source(y)}$.
\end{proof}

Here is an example of an $A$-cosieve 
that is not of the
form $\ppp_S$:

\begin{example}
Let $\C$ be the (2,1)-category 
of all sets with injective but not
surjective maps plus identities as 
1-cells and all 2-cells being 
identities. If $T$ is any set,
1-cells in $\C/T$ are given by
the relation
$x \preceq y \Leftrightarrow 
\mathrm{im}\, x \subseteq
\mathrm{im}\, y$; the 2-cells are
all identities. The group
$\mathrm{Aut}(T)$ is
trivial, 
all 1-cells in $\C$ are monic, 
and any $t \in T$ 
defines a cosieve 
$$
    x \ppp y : \Leftrightarrow 
    x \preceq y \mbox{ and } 
    t \in \mathrm{im}\, y \setminus 
    \mathrm{im}\, x,
$$
but
$S_\ppp$ (and hence
$\ppp_{S_\ppp}$) is empty, while 
$\bar S_\ppp$ consists of all
1-cells that are not identities.
\end{example}

Note, however, that 
in our main example,
$\ppp$ is of the form $\ppp_S$:

\begin{example}\label{abijan}
If $\C=\mfld^d$, then the set of
all embeddings $X \to Y$ of a
manifold $X$ into the interior of
$Y$
is a cosieve, hence $\ppp$ from
Example~\ref{polityka2} is a
$\mathrm{Diff}(T)$-cosieve 
given by a cosieve in
$\C$. 
\end{example}

\subsection{Orbit 2-categories}
We now show that the choice of an
$A$-cosieve $\pp$ and of 
an $A$-self-duality $\degree$ 
upgrades $\O_s$ to a (2,1)-category
which is \emph{2-thin} (contains at
most one 2-cell between any two
1-cells):

\begin{proposition}\label{modest}
Let $s \colon (S,\sqsubseteq) \to 
(U_G,\supseteq)$ be a $\G$-presheaf, 
$\degree$ be an $A$-self-duality,
and $\pp$ be an
$A$-cosieve. Then we have:
\begin{enumerate}
\item 
The relation
$$
    \gamma G_x \equiv \varepsilon
    G_x : \Leftrightarrow  
    (\forall u \in S \colon 
    u \pp y^\degree \Rightarrow 
    G_u \cap 
    \varepsilon G_x \gamma ^{-1} 
    \neq \emptyset)
$$
is an equivalence relation on
the morphism set $\O_s(x,y)$.
\item Interpreting $\equiv$ as a 2-cell
turns $\O_s$ into a (2,1)-category.
\item The $\G$-action on $\O_s$
induces an $\G$-action on
$\ho(\O_s)$.
\end{enumerate}
\end{proposition}

\begin{proof}
For all $u \in S$, we have
$1 \in G_u \cap \gamma G_x \gamma
^{-1}$, hence
$$ 
    \gamma G_x \equiv \gamma G_x.
$$
Next, if 
$ \alpha \in G_u \cap \varepsilon
G_x \gamma ^{-1}$, then 
$ \alpha ^{-1} \in G_u \cap 
\gamma G_x \varepsilon ^{-1}$,
so
$$ 
    \gamma G_x \equiv
\varepsilon G_x 
\Rightarrow \varepsilon G_x
\equiv \gamma G_x.
$$ 

Finally, if 
$ \alpha \in G_u \cap 
\varepsilon G_x \gamma ^{-1}$ 
and $ \beta \in G_u \cap 
\rho G_x \varepsilon ^{-1}$, then 
we obtain 
$ \beta \alpha \in G_u \cap 
\rho G_x \gamma ^{-1}$, so
$$ 
    \gamma G_x \equiv 
\varepsilon G_x,  
\varepsilon G_x \equiv   
\rho G_x \Rightarrow 
\gamma G_x \equiv \rho G_x.
$$ 
So $\equiv$ is an equivalence
relation and if we interpret it as
a 2-cell, there is
a (necessarily
unique and associative) vertical
composition of 2-cells (which in a
sense is induced by the product in
$G$), there is an identity
2-cell for each 1-cell $ \gamma
G_x$ (represented by 1), and 
all 2-cells are
invertible.  

Up to here no properties of
$\pp$ have been used. However,  
they are required to establish 
a horizontal
composition of 2-cells
$ \gamma G_x \equiv
\varepsilon G_x$ and 
$ \delta G_y \equiv \lambda G_y$
for two other 1-cells 
$ \delta G_y,\lambda G_y \colon 
y \to z$. 

We have to show 
$ \delta \gamma G_x \equiv 
\lambda \varepsilon G_x$.
To do so, assume $u \in S$ with 
$u \pp z^\degree$. Then as 
$ \delta G_y \equiv \lambda G_y$,
there exists an element 
$$
    \mu \in G_u \cap  
    \lambda G_y \delta ^{-1}.
$$
Since $ \varepsilon G_x $ is a
1-cell $x \to y$, we have 
$\varepsilon \gact x \sqsubseteq y 
\Rightarrow G_y \subseteq \varepsilon G_x
\varepsilon ^{-1}$, so we also have
\begin{equation}\label{madina}
    \mu \in G_u \cap  
    \lambda \varepsilon 
    G_x \varepsilon ^{-1}\delta ^{-1}.
\end{equation}
Furthermore, as $ 
\delta G_y $ is a 1-cell 
$ y\to z$, we have 
$ \delta \gact y \sqsubseteq z
\Rightarrow z^\degree \sqsubseteq 
\delta \gact y^\degree$. Thus 
$ u \pp z^\degree$ implies 
$u \pp \delta \gact y^\degree 
 \Rightarrow 
\delta ^{-1} \gact u
\pp y^\degree
$ and 
as $ \gamma G_x \equiv \varepsilon
G_x$, there exists some 
$$
    \alpha \in G_{\delta ^{-1}
\gact u} \cap \varepsilon G_x
\gamma ^{-1},  
$$ 
which means 
$$
    \delta \alpha \delta ^{-1} 
    \in G_u \cap 
    \delta \varepsilon G_x 
    \gamma ^{-1} \delta ^{-1}. 
$$
In combination with (\ref{madina})
we conclude 
$$
    \mu \delta \alpha \delta ^{-1} 
    \in G_u \cap \lambda \varepsilon G_x 
    \gamma ^{-1} \delta ^{-1},
$$
so this set is not empty as we had
to show in order to establish the
horizontal composition 
$ \delta \gamma G_x \To \lambda
\varepsilon G_x$
of 
$ \gamma G_x \To \varepsilon G_x$
with 
$ \delta G_y \To \lambda G_y$.

Due to the uniqueness of 2-cells,
the horizontal composition is
automatically
associative and satisfies the
exchange law. 

Last but not least, the action of $A$ on 
$\O_s$ defined in
Proposition~\ref{o2cat} descends to 
$\ho(\O_s)$, since for all $h \in
A$, we have
$$ 
    \gamma G_x \equiv \varepsilon
    G_x
    \Leftrightarrow 
 (h \gamma h^{-1} G_{hx} \equiv 
(h \varepsilon h^{-1}) G_{hx},
$$ 
simply conjugate all of 
$ G_u \cap \varepsilon G_x
\gamma^{-1} \neq \emptyset$ by
$h$. Similarly, the natural
transformations $ \gamma _x = 
\gamma G_x$ descend to $\ho(\O_s)$. 
\end{proof}

\begin{example}\label{herreger}
We keep extending
Examples~\ref{polityka} and
\ref{polityka2} with 
$\C=\mfld^d$, $T$ a compact
$d$-dimensional manifold without
boundary. 
Two 1-cells $ \gamma G_x , \varepsilon G_x 
\colon  x \to y$ in $\O_{s_{\C/T}}$ are represented by 
isotopy classes of isotopies 
$ \gamma \colon \mathrm{id} _T \To 
g, \varepsilon \colon \mathrm{id}
_T \To e$ with
$ \mathrm{im}\, gx \subseteq 
\mathrm{im}\, y, \mathrm{im}\, ex
\subseteq \mathrm{im}\, y$. To
distinguish this from the generic
case, we write here $\adw$ instead
of $\gact$; recall
that $ \gamma \adw x = gx,
\varepsilon \adw x = ex$, 
$ y \preceq z \Leftrightarrow 
\mathrm{im}\, y \subseteq
\mathrm{im}\, z$. The group $G_x$
contains the isotopies whose
restriction $ \gamma \WhiskR x$ to 
$ \mathrm{im}\, x$ is
(isotopic to) the constant isotopy 
with value $ \mathrm{id}
_{\mathrm{im}\, x}$. So we can
identify $ \gamma G_x$ via the
assignment
\begin{equation}\label{functoron1cells}
    \gamma G_x \mapsto \gamma
\WhiskR x \colon x \To gx
\end{equation}
with the restriction of 
$ \gamma \colon \mathrm{id} _T \To
g$ to $ \mathrm{im}\, x$. The 
coset $ \gamma G_x$ defines a 
1-cell $ \gamma G_x \colon x \to
y$ whenever $\mathrm{im}\, gx
\subseteq \mathrm{im}\, y$.
Now 
$ \gamma G_x \equiv \varepsilon
G_x$ means that for all
submanifolds 
$ \mathrm{im}\, u \subseteq T$ that
are disjoint from $ \mathrm{im}\,
y$ (i.e.~$u \ppp y^\degree$) there exists
a 2-cell $ \alpha \colon
\mathrm{id} _T \To a$ for some
diffeomorphism 
$a \colon T \to T$ which is on the
one hand in $ G_u$, that is, the
restriction of $ \alpha $ to 
$ \mathrm{im}\, u$ is constantly
equal to the identity,
$$
    \alpha (t,p) = p \quad \forall
    t \in [0,1],p \in \mathrm{im}\,
    u \subset T \setminus
\mathrm{im}\, y. 
$$ 
At the same time, $ \alpha 
\in \varepsilon G_x \gamma ^{-1} = 
\varepsilon \gamma ^{-1} G_{\gamma
\gact x}$ means that $ \alpha x$ is
a 2-cell that composes with 
$ \gamma x$ to $ \varepsilon x$.  
The upshot is that $ \gamma G_x
\equiv \varepsilon G_x$ means that
for any choice of submanifold 
$ \mathrm{im}\, u $ that is
disjoint from $ \mathrm{im}\, y$,
we find an isotopy that fixes 
$ \mathrm{im}\, u$ pointwise and
deforms the embedding 
$ gx \colon X \to T$ inside the
complement of $ \mathrm{im}\, u$ to 
$ ex \colon X \to T$.
\end{example}

\subsection{Lifting of
self-dualities}
Let $s \colon (S,\sqsubseteq) 
\rightarrow (U_G,\supseteq)$ be a 
$\G$-presheaf and 
$(S,\sqsubseteq)$ be
self-dual. 
If $ \gamma G_x \colon x \to y$ is
a morphism in $\O_s$, then 
$ \gamma \gact x \sqsubseteq y$. 
As $\degree$ is a self-duality, 
$ y^\degree \sqsubseteq 
(\gamma \gact x)^\degree = 
\gamma \gact x^\degree$, hence
$ \gamma ^{-1} \gact y^\degree
\sqsubseteq x^\degree$ and there is
a morphism $ \gamma ^{-1}
G_{y^\degree} \colon y^\degree \to
x^\degree$. However, in general
this does not lead to a
self-duality of 
$\O_s$ itself: $ \gamma G_x \mapsto
\gamma ^{-1} G_{y^\degree}$ is not
well-defined unless $ G_x \subseteq 
G_{y^\degree}$. However, on
$\ho(\O_s)$ we obtain a somewhat
satisfactory resolution.

One easily verifies
that $ \gamma G_x= \varepsilon
G_x$ implies $ \gamma ^{-1}
G_{y^\degree} \equiv \varepsilon
^{-1} G_{y^\degree}$, but in general,
$ \gamma G_x \equiv \varepsilon
G_x$ does not necessarily imply
$ \gamma^{-1} G_{y^\degree} \equiv 
\varepsilon^{-1} G_{y^\degree}$. 
We now formulate a technical
condition that ensures it is; we
will explain in
Example~\ref{tubular} that this is an abstract
replacement of the
existence of a tubular neighbourhood
of a submanifold.

\begin{proposition}\label{lome}
Assume that $\degree$ is an 
$A$-self-duality, $\pp$ is
an $A$-cosieve, 
and that for all
$b,c,d \in S$ we have
\begin{equation}
\begin{aligned}\label{tubucondi}
    &
(c \pp b^\degree
    \mbox{ and } d \pp b^\degree)
    \\
\Rightarrow \quad &
    \exists \rho \in G_c \cap G_d,  
    a \in S :  
    (\rho \gact a \sim b
    \mbox{ and }
    a \pp b). 
\end{aligned}
\end{equation}
If 
$ \gamma G_x \equiv \varepsilon G_x
\colon x \to y$ in $\O_s$, then we have 
$\gamma ^{-1} G_{y^\degree}
\equiv \varepsilon ^{-1}
G_{y^\degree} \colon y^\degree 
\to x^\degree$. 
\end{proposition}

\begin{proof}
We need to show 
\begin{align*}
    \gamma^{-1} G_{y^\degree} \equiv 
    \varepsilon^{-1} G_{y^\degree} 
& \Leftrightarrow     
    \forall v \pp x : 
    G_v \cap \varepsilon^{-1} G_{y^\degree} 
    \gamma  
    \neq \emptyset \\
& \Leftrightarrow     
    \forall v \pp x : 
    \varepsilon G_v \gamma ^{-1} \cap 
    G_{y^\degree} 
    \neq \emptyset.
\end{align*}
We are going to apply (\ref{tubucondi}) with 
$$
    b = y^\degree ,\quad
    c = \gamma \gact v,\quad
    d = \varepsilon \gact v, 
$$
so we need to show 
$ 
(\gamma \gact v)
\pp 
y^{\degree\degree} 
$ 
and 
$ 
(\varepsilon \gact v)
\pp
y^{\degree\degree}$. 
To do so,  
recall once more that
$ \gamma G_x, \varepsilon G_x
\colon x \to y$ are 1-cells, so
$
(\gamma \gact x) 
\sqsubseteq y
,
(\varepsilon \gact x)
\sqsubseteq y
$. 
Together with 
$ v \pp x \Rightarrow 
\gamma \gact v \pp 
\gamma \gact x$ and 
$ v \pp x \Rightarrow 
\varepsilon \gact v \pp 
\varepsilon \gact x$
we conclude $\gamma
\gact v \pp y, \varepsilon
\gact v \pp y$, and now we can use 
$y^{\degree\degree} \sim y$ and 
Lemma~\ref{lucy}.  

So by (\ref{tubucondi}), there are 
$
    \rho \in G
$,
$a \in S$ satisfying 
$$
    \rho \in G_{\gamma \gact v}
    \cap G_{\varepsilon \gact
v},\quad
    \rho \gact y^\degree
    \sim a,\quad
    a \pp y^\degree.
$$
Now we use 
$ \gamma G_x \equiv \varepsilon
G_x$ with $u:=a$. This shows that 
$$
    G_a \cap \varepsilon G_x \gamma
^{-1} = 
    G_{\rho \gact y^\degree} \cap 
    \varepsilon G_x \gamma 
    \neq \emptyset
$$
which implies (conjugate with $
\rho ^{-1}$ and use 
$v \pp x \Rightarrow v
\sqsubseteq x \Rightarrow G_x \subseteq
G_v$) 
$$
    G_{y^\degree} \cap 
    \rho ^{-1} \varepsilon G_v \gamma
    ^{-1} \rho \neq \emptyset.
$$
As $ \rho $ and hence $ \rho ^{-1}$
is both in $ G_{\gamma \gact v}$
and 
$ G_{\varepsilon \gact v}$ we
finally have
\begin{align*}
    \rho ^{-1} \varepsilon G_v \gamma
    ^{-1} \rho 
&=
    \rho ^{-1} G_{\varepsilon \gact
x} 
    \varepsilon \gamma^{-1} \rho \\
&=
    G_{\varepsilon \gact v} 
    \varepsilon \gamma^{-1} \rho \\
&=
    \varepsilon \gamma ^{-1} 
    G_{\gamma \gact v} \rho \\
&=
    \varepsilon \gamma ^{-1} 
    G_{\gamma \gact v} \\
&=
    \varepsilon 
    G_v \gamma^{-1}.\qedhere
\end{align*} 
\end{proof}

\begin{corollary}\label{modest2}
Under the assumptions of
Proposition~\ref{lome}, $\ho(\O_s)$
is a self-dual category, with 
the dual of $ [\gamma G_x] \colon
x \to y$ given by
$$
    [\gamma G_x]^\degree :=
    [\gamma ^{-1} G_{y^\degree}]
\colon  y^\degree \to x^\degree. 
$$
\end{corollary}
\begin{example}\label{tubular}
For $s_{\mfld^d/T}$, 
(Example~\ref{herreger}), 
$b,c,d$ correspond to submanifolds 
$ B:=\mathrm{im}\, b, C:=\mathrm{im}\, c,
D:=\mathrm{im}\, d \subseteq T$ with 
$$ 
    B \cap C = B \cap D = \emptyset ,
$$
so $C,D$ are both contained
in the interior of $T\setminus B$
(as $ c \pp b^\degree, 
d \pp b^\degree$). 
Condition~(\ref{tubucondi}) asserts
the existence of an isotopy 
$ \rho \colon \mathrm{id} _T \To r$ 
for some diffeomorphism $r \colon T
\to T$ such that $ \rho c $ is
constantly $ \mathrm{id} _C$ and 
$ \rho d$ is constantly $
\mathrm{id} _D$ while $ \rho $
shrinks $B$ to a
manifold $A = \mathrm{im}\, a $
contained in the interior of $B$
which is however diffeomorphic to 
$B$ via $r$. To obtain such an
isotopy, choose a Riemannian metric
on $T$. Extend the outward normal
vectors of length 1 on $B \subset
T$ to a vector field on all of $T$
that is only supported on a small
neighbourhood of $\partial B$ 
disjoint from $C$ and $D$
(using e.g.~a partition of 1 and 
bump functions). Following the
inverse flow
of this vector field for times in a
sufficiently small closed time
interval yields $ \rho $ (or rather 
$ \rho ^*$).
\end{example}

\section{Application to slice
2-categories}\label{assumptionssec}
Our main theorem follows
more or less immediately from the results above, but
we discuss its assumptions and
our main example in more detail,
as well as the canonical choice of 
the relation $\ppp$.  
Throughout, we make
Assumptions~\ref{eins}
and~\ref{zwei}, and  
$A=\mathrm{Aut}(T),G,G_x,\adw,\preceq$ and $s_{\C/T}$ 
are as in 
Examples~\ref{elmena} and
\ref{herreger}.

\subsection{The interior of a
subobject}
The theory developed in
Section~\ref{4advent} crucially
relies on the
choice of an $\Aut$-cosieve with 
certain properties. This is an
auxiliary structure though, neither
the category $\ho(\C/T)$ nor the
resulting self-duality on it
depends on this choice.

As we will discuss in the next
subsection, the central assumption of
our theorem is a strong form of the
homotopy extension property
well-studied in algebraic and
differential topology. We  
now define an $\Aut$-cosieve 
$\ppp$ that is adapted to 
the formulation of this assumption:

\begin{definition}\label{defcos}
If $u \preceq v$, then we set
\begin{align*}
    u \ppp v  
    \> : \Leftrightarrow \>
    \forall \xi \colon x \To z 
    \,\exists \gamma \in G_u : 
    \target(x) = \source
    (v^\degree) \Rightarrow 
    \gamma v^\degree x  = 
    v^\degree \xi.
\end{align*}
\end{definition} 

As we will explain in
Example~\ref{hepgood} below, this is
consistent with the notation
introduced for the special case 
$\mfld^d/T$ in
Example~\ref{polityka2}
above. The interpretation derived
from this example is that a subobject $[u]$
of $T$ that is contained in a
subobject $[v]$ is in the
\emph{interior} of $[v]$ if and only
if any 2-cell $\xi$ which only
acts in the complement
$[v^\degree]$ of $[v]$ can be
extended to a 2-cell 
$ \gamma \colon \mathrm{id} _T 
\To g$ for which $g \in \Aut$ 
and $ \gamma $ is constantly the identity 
on $[u]$,  $ \gamma u = u$.

Let us verify that
Definition~\ref{defcos} 
indeed defines an $\Aut$-cosieve:

\begin{proposition}\label{pppiscos}
$\ppp$ is an $\Aut$-cosieve
in $(\C/T_0,\adw,\preceq)$. 
\end{proposition}
\begin{proof}
Assume $u \ppp v$. Then 
$u \preceq v$
holds by definition, say $u=vb$. If 
$ d \in \Aut$, then
we have  
$ du \preceq dv$ as $du=dvb$.
Also, the domain of $ (d v)^\degree = 
d v^\degree$ agrees with the domain
of $v^\degree$. Hence if $\xi
\colon x \To z$ is any 2-cell in
$\C$ between 1-cells $x,z$ whose 
codomain is $\target(x) = \target
((d v)^\degree = 
\target (v^\degree)$, then by
assumption, there exists $ \gamma
\in G$ with $ \gamma u = u $ and 
$ \gamma v^\degree x = v^\degree
\xi$. Then $ \eta := d \gamma d^{-1}$ 
satisfies 
$$ 
    \eta (du) = d \gamma d^{-1}
    du = d \gamma u = d u
$$
and 
$$
    \eta (dv)^\degree x = \eta (d v^\degree) x
=   d \gamma d^{-1} d v^\degree x =
    d \gamma v^\degree x = 
    d v^\degree \xi = 
    (dv)^\degree \xi.
$$
Thus $ du \ppp 
dv$. Similarly, if $v
\preceq w $, say $v = wc$, then 
we have $u = vb=wcb \preceq w$
and
$w^\degree \preceq v^\degree$, say 
$w^\degree =v^\degree l$. 
Furthermore, if $ \psi \colon 
m \to n$ is a 2-cell between
1-cells $m,n$ with codomain 
$$
    \target(m) = \source (w^\degree) =
\source (l),
$$ 
then $\xi := l \psi \colon x \To
z$,
$x :=lm,z := ln$, is a 2-cell and the target 
$\target(x) = \source
(v^\degree)$, so there exists 
$ \gamma \in G_u$ with 
$$
    \gamma w^\degree m 
    = \gamma v^\degree lm
    = \gamma v^\degree x = 
    v^\degree \xi = v^\degree l
\psi = w^\degree \psi. 
$$ 
So $u \ppp w$ as we had to show. 
\end{proof}

So far, this subsection has not
made any assumptions on 
$\C$ and $T$, but what we need to
demand is that the
converse of Definition~\ref{defcos}
holds, that is, that we can
characterise the image of 
the maps $ \C_2 \to \C_2,\xi
\mapsto e \xi$ in terms of 
$\ppp$:

\begin{assumption}\label{seven}
If $x,z \colon X \to E$, $e \colon
E \to T$, and 
$ \phi \colon ex \To ez$,
then 
$$  
  (\exists \xi \colon x \To z : 
    \phi = e \xi)
 \Leftrightarrow 
    (\forall u \ppp e^\degree 
    \exists \gamma \in G_u : 
    \gamma ex = 
    \phi). 
$$
\end{assumption}

To be clear: $ \Rightarrow $ holds
by definition of $\ppp$, what we
assume is $ \Leftarrow$.

Assumption~\ref{seven} enters the
proof of our theorem in the
following way:

\begin{proposition}\label{immerbesser}
There exists a 2-cell 
$\xi$ between 1-cells 
$ \gamma x, 
\varepsilon x$ in 
$\C/T$ with $x \in \C/T_0, 
\gamma , \varepsilon \in G$ 
if and only if 
$ \gamma G_x \equiv \varepsilon
G_x$. 
\end{proposition}
\begin{proof}
If $ \gamma \colon 
\mathrm{id} _T \To g$, then we have
\begin{align*}
\varepsilon x = y \xi \circ
\gamma x 
& \Leftrightarrow 
\varepsilon x \circ (\gamma x)^* =
y \xi \\
& \Leftrightarrow 
    \varepsilon x \circ \gamma ^*
    x= 
    y \xi \\
& \Leftrightarrow (\varepsilon
\circ \gamma ^*) x = y \xi \\
& \Leftrightarrow 
    (\varepsilon \circ g \gamma
^{-1} ) x = y \xi \\
& \Leftrightarrow \varepsilon  
    g \gamma ^{-1} x = y \xi \\
& \Leftrightarrow \varepsilon  
    \gamma ^{-1} gx = y \xi. 
\end{align*}
Now apply Assumption~\ref{seven} 
with $\phi = \varepsilon \gamma
^{-1} gx$.
\end{proof}

\subsection{The homotopy extension property}
A priori, the $\Aut$-cosieve $\ppp$
could be empty. As we
will show now, the backbone of our
main theorem
is that the objects in $\C/T_0$ are
\emph{cofibrations}, which
in the language we have developed
means that all objects have a
nonempty interior: as
$ \mathrm{id} _T$ is terminal
in 
$(\C/T_0,\preceq)$, 
$\mathrm{id} _T^\degree$ is
initial, that is, we have
$ \mathrm{id} _T^\degree \preceq 
v$ for all $v \in \C/T_0$, so
the following implies 
$ \mathrm{id} _T^\degree \ppp
v$ for all $v$.

\begin{assumption}\label{fortytwo}
$ \mathrm{id} _T^\degree \ppp
\mathrm{id} _T^\degree$.
\end{assumption} 

So explicitly, we 
assume that for all $x,z \in \C/T_0$,
we have
\begin{equation}\label{fortythree}
    \forall \xi \colon x \To z 
    \,\exists \gamma \in G_{
    \mathrm{id} _T^\degree} : 
    \gamma x  = 
    \xi.
\end{equation}

This implies:

\begin{proposition}
If $ g \hsim \mathrm{id} _T$, then 
$g$ is invertible. 
\end{proposition}
\begin{proof}
If $ \xi \colon \mathrm{id} _T \To
g$ is a 2-cell, then by
Assumption~\ref{fortytwo}, 
$ \xi \in G_{\mathrm{id}
_T^\degree} \subseteq G$, so its
codomain is invertible. 
\end{proof}

More importantly, if
$ \gamma G_x \colon x \to y$ 
is a morphism in $\O_{s_{\C/T}}$, then by 
definition, we have 
$\gamma \adw x \preceq y$, 
so there exists 
a unique (all 1-cells are monic) 
$ f \colon x \to y$ 
with $ \gamma \adw x = yf$, and with 
$ \phi := \gamma \WhiskR x$ we
obtain a morphism 
$x \to y$ in $\C/T$, viewed as an
ordinary category. 
Thus we obtain a functor 
$$
    \F \colon 
    \O_{s_{\C/T}} 
    \rightarrow \C/T,\quad
    (\gamma G_x \colon 
    x \to y) \mapsto 
    ((f,\gamma \WhiskR x) 
    \colon x \to
    y)
$$
which is the identity map on
objects and is easily seen to be
compatible with the
$\aut(T)$-actions. 
It is by the definition of
$G_x$ faithful, and by
(\ref{fortythree}) and
Proposition~\ref{immerbesser}, we in fact
finally obtain:

\begin{proposition}\label{vladimir}
The functor $\F$ is an isomorphism 
and induces
an isomorphism 
of $\aut(T)$-categories 
$\ho(\C/T) \cong 
\ho(\O_{s_{\C/T}})$.  
\end{proposition}

More precisely, ``$ \Rightarrow $''
in Assumption~\ref{seven} implies
that $\F^{-1}$ extends to a
2-functor $\C/T \to \O_{s_{\C/T}}$
and if in addition ``$
\Leftarrow$'' holds, this induces
an isomorphism
$
    \ho(\C/T) \cong 
    \ho(\O_{s_{\C/T}}). 
$

\begin{remark}
A 1-cell $x \colon X \to T$ in a
2-category is a
\emph{cofibration} (see
e.g.~\cite{MR4177953}) if for
all $\phi \colon s x \To v, 
s \colon T \to V,v \colon X \to V$
there exists $w \colon 
T \to V$ with $v=wx$ and 
$\rho \colon s \To w$ 
with $ \phi = \rho \WhiskR x$.
As we assume all 1-cells are monic, 
$\C/T$ remains unchanged if we
discard all
objects $V$ in $\C$ without a
1-cell $V \to T$, and 
then Assumption~\ref{fortytwo}
simply says that all 1-cells in 
$\C$ are cofibrations. 
In many examples, there is a 
\emph{path space object} $V^I$ in 
$\C$ that comes
equipped with source and target
1-cells 
$ \source,\target \colon
V^I \to V$ 
(think
of a space of paths $[0,1] \to V$ in a
space $V$), and 2-cells 
$ \phi \colon r \To v$ between
1-cells $r,v \colon X \to V$ 
correspond to 1-cells  
$ p \colon X \to V^I$ with  
$ r=\source p,v=\target p$.
Then the cofibration property 
becomes depicted by the 
following standard diagram 
\begin{equation*}
    \begin{tikzcd}[]
        X
            \ar[rr, "p"]
            \ar[dd, "x"{swap}]
            & &
        V^I
            \ar[dd, "\partial_0"]
            \\ \\
        T
            \ar[rr, "s"]
            \ar[uurr, dashed,
"\exists l"]
& &
		V.
    \end{tikzcd}
\end{equation*}
\end{remark}

\begin{example}\label{hepgood}
In $\C=\mfld^d$,
$ \mathrm{id} _T^\degree$ is the
empty embedding of the empty set,
and 
Assumptions~\ref{seven} and 
\ref{fortytwo} hold if and
only if $T$ has no boundary. 
We refer e.g.~\cite[Theorem~8.1.6]{MR1336822} 
for the
proof of the homotopy extension
property; see also
the recent post 
\cite{Goodwillie2018} by Goodwillie who
discusses the uniqueness of the
extensions. Reformulated in our
language, he therein points out
that $G_x$ and hence 
$ \gamma G_x$ is not just a
path-connected, but a contractible
topological space. In this sense,
the extension $ \gamma $ of a given
$ \phi $ to all of $ T$ is from a
homotopy-theoretic point of view
unique. The non-uniqueness of the
extension was the reason why we
introduced the auxiliary tool of
the orbit 2-category
$\O_{s_{\C/T}}$. Its
1-cells are the cosets 
$ \gamma G_x$ rather than the
representatives $ \gamma $
themselves, and that is why $\F$ is 
faithful by definition. 

Once it is established that 
2-cells between $x,z \colon X \to
T$ extend to all of $T$, it is 
easily seen that 
Assumption~\ref{seven} holds
and that $u \ppp v$ as defined in
Definition~\ref{defcos} 
means that $\mathrm{im}\, u$ is contained
in the interior of $ \mathrm{im}\,
v$ (Example~\ref{polityka2}):
indeed, if we have 
$$ 
\mathrm{im}\, x, \mathrm{im}\, z
\subseteq E \subsetneq T,\quad
    U = \mathrm{im}\, u 
\subsetneq V=T \setminus E,
$$
then there is a tubular
neighbourhood of $ \partial E$ that
is disjoint from $U$. Using a
partition of 1 we obtain a smooth
bump function $b \in C^\infty (T)$ 
with value $1$ on $E$ and value $0$
on $U$; now any extension $\eta \in
G$ of 
an isotopy $\phi \colon x \To z$
can be replaced by $ \gamma \in
G_u$ given by $$ 
    \gamma (t,p) := 
    \eta (t b(t), p),\quad
    t \in [0,1],p \in T
$$ 
with $ \gamma (t, p) =p$ for all 
$t$ and $p \in U$ (as in
Example~\ref{herreger}).
Conversely, if such an extension
of $ \phi $ exists, then 
$ \phi $ only acts inside 
$T \setminus U$. So if $U$ can be
chosen arbitrarily in $T\setminus
E$, $ \phi $ is of
the form $ e \xi$. 
\end{example}

\subsection{Tubular neighbourhoods} 
To complete the proof of
our theorem, we need to show
that $\degree$ extends to 
$\O_{s_{\C/T}}$, and here we simply
assume outright the required
condition (\ref{tubucondi}):
\begin{assumption}\label{last}
For all
$b,c,d \in \C/T_0$ with 
$c \ppp b^\degree$
and $d \ppp b^\degree$,
there are 
$ \rho \in G_c \cap G_d$ and 
$a \in \C/T_0$ with 
$\rho \adw a \sim b$
and $a \ppp b$. 
\end{assumption}

The picture we have in mind has
already been discussed in
Example~\ref{tubular}: we are given
submanifolds $B \subseteq T$
and $C,D \subset T \setminus B$ and
then use a tubular neighbourhood of 
$ \partial B$ to replace $B$ by a
slightly smaller but diffeomorphic
submanifold $A \subset B$. So
Assumption~\ref{last} is about an
abstract form of tubular
neighbourhoods and deformation
retracts.  

Nowhere in this paper
we have assumed that the preorders
studied are lattices,
i.e.~that one can take some form of 
unions or intersections of
subobjects, and in fact in the case
of $\C=\mfld^d$, one can not.
However, the union of submanifolds 
$C,D$ contained in the interior of
a submanifold is obviously
contained in a slightly larger
submanifold $E$, and if $\C/T$ has
this property, then we can use the
homotopy extension property to
formulate Assumption~\ref{last}
internally in $B$:

\begin{example}
Assume that in $\C/T$, there exists  
for all $c \ppp v,d \ppp v$ an $e
\ppp v$ with $c \preceq v, d
\preceq v$. Them
Assumption~\ref{last} can be
reduced to the assumption that for
all $b \colon B \to T$ there exists 
an invertible 1-cell $i \colon B
\to B$ with $a:=bi \ppp b$ and 
$ i \hsim \mathrm{id} _B$, as 
by Assumption~\ref{seven} a 2-cell
$ \iota \colon i \To \mathrm{id}
_B$ gives rise to 
$ b \iota \colon a \To b$ which
can be extended to 
a 2-cell $ \rho \in G_e \subseteq G_c \cap
G_d$ as in Assumption~\ref{last}. 
\end{example}
 
To sum up: the assumptions made in
the present section are those from
our main theorem, which follows from 
Proposition~\ref{vladimir}
(Assumptions~\ref{seven}
and~\ref{fortytwo} imply
$\ho(\C/T) \cong \ho(s_{\C/T})$) 
in combination with 
Corollary~\ref{modest2} 
(Assumption~\ref{last} implies that
$\ho(s_{\C/T})$ is self-dual).


\end{document}